\documentclass[10pt,a4paper,reqno]{amsart}
\usepackage{amsfonts}
\usepackage{amsthm}
\usepackage{amsmath}
\usepackage{mathtools}
\usepackage{amscd}
\usepackage[utf8]{inputenc}
\usepackage{t1enc}
\usepackage[mathscr]{eucal}
\usepackage{indentfirst}
\usepackage{graphicx}
\usepackage{graphics}
\usepackage{esint}
\usepackage{pict2e}
\usepackage{epic}
\usepackage{float}
\usepackage{MnSymbol}
\usepackage{multirow}
\usepackage[table,xcdraw]{xcolor} 
\usepackage{enumitem} 
\usepackage[nocompress]{cite}
\usepackage{hyperref}
\usepackage{stmaryrd}
\usepackage{todonotes}
\usepackage{tcolorbox}
\numberwithin{equation}{section}
\usepackage[margin=2.9cm]{geometry}
\usepackage{epstopdf}
\usepackage{xargs}
 
\allowdisplaybreaks

\newmuskip\pFqmuskip
\newcommand*\pFq[6][8]{%
	\begingroup 
	\pFqmuskip=#1mu\relax
	\mathchardef\normalcomma=\mathcode`,
	\mathcode`\,=\string"8000
	\begingroup\lccode`\~=`\,
	\lowercase{\endgroup\let~}\pFqcomma
	{}_{#2}F_{#3}{\left[\genfrac..{0pt}{}{#4}{#5};#6\right]}%
	\endgroup
}
\newcommand{\pFqcomma}{{\normalcomma}\mskip\pFqmuskip}

\theoremstyle{plain}
\newtheorem{theorem}{Theorem}[section]
\newtheorem{lemma}[theorem]{Lemma}
\newtheorem{corollary}[theorem]{Corollary}
\newtheorem{proposition}[theorem]{Proposition}

\theoremstyle{definition}

\newtheorem{remark}[theorem]{Remark}

\newtheorem{example}[theorem]{Example}

\newcommand{\Res}[1]{\underset{#1}{\operatorname{Res}} \ }

\newcommand{\A}{\mathcal A}
\renewcommand{\L}{\mathcal L}
\newcommand{\y}{\boldsymbol{y}}
\newcommand{\z}{\boldsymbol{z}}
\newcommand{\Ch}{\operatorname{Ch}}
\newcommand{\ind}{\mathbf 1}
\renewcommand{\Res}[1]{\underset{#1}{\operatorname{\: Res \:}}}

\newcommand{\Frac}{\operatorname{Frac}}
\newcommand{\m}{\boldsymbol{m}}
\newcommand{\FP}{\operatorname{FP}}
\renewcommand{\Re}{\operatorname{Re}}

\newcommand{\e}{\mathbf{e}}

\newcommandx{\change}[2][1=]{\todo[linecolor=blue,backgroundcolor=blue!25,bordercolor=blue,#1]{#2}}

\title[Hyperplane arrangements and the Witten zeta function]{Combinatorics of hyperplane arrangements and Witten zeta function at the origin}
\author{Kam Cheong Au}
\address[Kam Cheong Au]{University of Cologne, Department of Mathematics and Computer Science, Weyertal 86-90, 50931 Cologne, Germany}
\email{kau@uni-koeln.de}
\author{Kazuhiro Onodera}
\address[Kazuhiro Onodera]{Division of Mathematics, Chiba Institute of Technology, 2-1-1 Shibazono, Narashino, Chiba 275-0023, Japan}
\email{onodera@chibatech.ac.jp}

\date{\today}

\subjclass[2020]{Primary: 11M32, 17B22, 52C35. Secondary: 11M41, 32S22}

\keywords{Cone valuation, hyperplane arrangement, M\"obius function, root system, special value, Witten zeta function}

\begin{document}
	
	\maketitle
	
	\begin{abstract} We introduce a new method that brings the combinatorics of hyperplane arrangements into the study of representation zeta functions of compact Lie groups. For the Witten zeta function $\zeta_\Phi(s)$ associated with a root system $\Phi$, our method yields elegant formulas for $\zeta_\Phi(0)$ and $\zeta_\Phi'(0)$ in terms of the exponents of various parabolic subsystems of $\Phi$. Such formulas do not appear to be readily accessible through the conventional analytic techniques in the literature. More generally, the method applies to a broad family of conical zeta functions, expressing these two special values through the M\"obius function of the intersection poset of the associated hyperplane arrangement. 
	\end{abstract}
	
	\section{Introduction}
	Let $G$ be a compact Hausdorff topological group. Its \textit{representation zeta function} \cite{larsen2008representation} is defined by $$\zeta_G(s) := \sum_{\rho} \frac{1}{(\text{dim} \: \rho)^s},$$ where $\rho$ ranges over the finite-dimensional irreducible complex representations of $G$.
	When $G$ is a simply-connected compact Lie group with crystallographic root system $\Phi$, we write $\zeta_\Phi(s)$ for $\zeta_G(s)$; this is the \textit{Witten zeta function} associated with $\Phi$. For $G=\text{SU}(2)$, corresponding to $\Phi=A_1$, one has $\zeta_\Phi(s)=\zeta(s)$, the classical Riemann zeta function.
	
	Witten zeta functions have been studied from several perspectives, beginning with work of Witten \cite{witten1991quantum} and Zagier \cite{zagier1994values}. The analytic properties of their multivariable generalizations have received extensive investigation \cite{komori2023theory,matsumoto2006witten,komori2010witten,onodera2014functional, komori2015witten,komori2012witten,komori2020zeta,nakamura2006functional}. More recently, the first author uncovered a range of their arithmetic properties in \cite{au2024single,au2024vanishing}, including $p$-adic Kummer congruences, special values of residues, and special values at the origin and at negative integers; see also \cite{matuzas2026second, matuzas2026universal}. The Witten zeta function also occurs as the Archimedean factor of representation zeta functions of arithmetic subgroups; these zeta functions have received considerable attention \cite{avni2013representation, stasinski2014representation, blomer2026analytic}.
	
	A main goal of this article is to evaluate $\zeta_\Phi(0)$ and $\zeta_\Phi'(0)$. We will see that they encode non-trivial quantities associated to the root system $\Phi$. These two numbers also arise in asymptotic formulas for a family of partition functions associated with $G$. Such formulas generalize the Hardy--Ramanujan asymptotic formula for the partition function and have been investigated by several authors \cite{hardy1918asymptotic,bridges2024asymptotic,bridges2023number,bringmann2023asymptotic,debruyne2020saddle,romik2017number}.

	Weyl's dimension formula expresses $\zeta_\Phi(s)$ as certain Shintani-type multiple zeta functions. For example,
	\begin{align}\label{witten_zeta_explicit_ex}
		\zeta_{A_2}(s) &= 2^s \sum_{m_1,m_2\geq 1} \left(m_1 m_2 (m_1+m_2)\right)^{-s},\nonumber \\
		\zeta_{B_2}(s) &= 6^s \sum_{m_1,m_2\geq 1} \left(m_1 m_2 (m_1+m_2) (m_1+2m_2)\right)^{-s}, \\
		\zeta_{A_3}(s) &= 12^s \sum_{m_1,m_2,m_3\geq 1}  \left(m_1 m_2 m_3 (m_1+m_2) (m_2+m_3) (m_1+m_2+m_3)\right)^{-s} \nonumber.
	\end{align}
	Such expressions have been used to calculate $\zeta_\Phi(0)$ and $\zeta_\Phi'(0)$ for several root systems of small rank through \textit{ad hoc} case-by-case computations \cite{au2024single,rutard2026values,romik2017number,bailey2018computation,onodera2014functional,onodera2018generalized,onodera2021multiple}. Although the formulas in \eqref{witten_zeta_explicit_ex} are appealing, they lack intrinsic structure that enables a systematic analysis.
	
	By using the combinatorial theory of hyperplane arrangements, we recast Witten zeta functions in a more intrinsic form. This novel viewpoint yields elegant evaluations of $\zeta_\Phi(0)$ and $\zeta_\Phi'(0)$ for root systems $\Phi$ of arbitrary rank. Our first main result evaluates $\zeta_\Phi(0)$. 
	\begin{theorem}\label{thm:witten_0}
		Let $\Phi$ be a root system of rank $r$ with Weyl group $W$, and let $E(\Phi)$ be the product of the exponents of $W$. The value of the Witten zeta function at the origin is given by $$\zeta_\Phi(0) = (-1)^r \frac{E(\Phi)}{|W|}.$$
	\end{theorem}
	
	Our second and deeper main result evaluates $\zeta_\Phi'(0)$ in terms of codimension-one parabolic root systems. It is most naturally stated in terms of the normalized function
	$$\xi_\Phi(s) := K_\Phi^{-s} \zeta_\Phi(s)$$
	where $K_\Phi \in \mathbb{N}:=\{1,2,\ldots\}$ is an explicit integer; see equation \eqref{aux_5}. We first recall some standard notation. Let $\Delta = \{\alpha_1,\ldots,\alpha_r\}$ be a set of simple roots of $\Phi$, and let $\Phi^+$ be the set of positive roots determined by $\Delta$. With respect to an ambient inner product $(\cdot,\cdot)$, write $\gamma^\vee:= 2\gamma/(\gamma,\gamma)$, and let $\{\lambda_1,\ldots,\lambda_r\}$ be the corresponding fundamental weights. For $J\subset \Delta$, denote by $\Phi_J$ the parabolic root system induced by $J$ and by $W_J$ its Weyl group. 
	
	\begin{theorem}\label{thm:witten_0_der}
		For each positive integer $d$, we define $$m_\Phi(d) := \sum_{i=1}^r \frac{E(\Phi_{\Delta-\{\alpha_i\}})}{|W_{\Delta-\{\alpha_i\}}|} \#\{ \gamma\in \Phi^+ : (\lambda_i, \gamma^\vee) = d\}\in \mathbb{Q}.$$
		Then\footnote{Note that $m_\Phi(d) \neq 0$ for only finitely many $d$, so the next displayed sum is finite.}
		$$\xi_\Phi'(0) = (-1)^r \frac{|\Phi^+| E(\Phi)}{|W|} \log(2\pi) + \frac{(-1)^{r+1}}{2} \sum_{d\geq 1} m_\Phi(d) \log d.$$
	\end{theorem}
	
	This formula immediately yields the explicit values displayed in Table \ref{table:Witten-zeta-values-at-zero} below. It also confirms a conjecture of the first author in \cite{au2024single}:
	\begin{equation}\label{bad_primes_Weyl}\xi_\Phi'(0) \in   \begin{cases}
			\mathbb{Q} \log(2\pi)  \qquad &\text{ for }\Phi = A_r, \\
			\mathbb{Q} \log(2\pi) + \mathbb{Q} \log (2)  \qquad &\text{ for }\Phi = B_r, C_r, D_r,\\
			\mathbb{Q} \log(2\pi) + \mathbb{Q} \log (2) + \mathbb{Q} \log (3) \qquad &\text{ for }\Phi = G_2, F_4, E_6, E_7,\\
			\mathbb{Q} \log(2\pi) + \mathbb{Q} \log (2) + \mathbb{Q} \log (3) + \mathbb{Q} \log (5) \qquad &\text{ for }\Phi = E_8.
	\end{cases}\end{equation}
	Later in this introduction, we give some interpretations of the primes $p$ whose logarithms occur in the expression above.
	
	\begin{table}[!t]
		\centering
		\small
		\renewcommand{\arraystretch}{1.65}
		\setlength{\tabcolsep}{7pt}
		\resizebox{0.8\textwidth}{!}{%
			\begin{tabular}{|c|c|c|}
				\hline
				$\Phi$ & $(-1)^{\operatorname{rank}\Phi} \; \xi_\Phi(0)$
				& $(-1)^{\operatorname{rank}\Phi} \; \xi_\Phi'(0)$ \\
				\hline
				$A_r$
				& $\displaystyle \frac{1}{r+1}$
				& $\displaystyle \frac{r}{2}\log(2\pi)$ \\ 
				$B_r$
				& $\displaystyle \frac{a_r}{2r}$
				& $\displaystyle \frac{ra_r}{2}\log(2\pi)-
				\left(\frac{r+2}{6}a_r - \frac{r+1}{4}\right)\log 2$ \\ 
				$C_r$
				& $\displaystyle \frac{a_r}{2r}$
				& $\displaystyle \frac{ra_r}{2}\log(2\pi)-
				\frac{r-1}{6}a_r\log 2$ \\ 
				$D_r$
				& $\displaystyle \frac{a_{r-1}}{2r}$
				& $\displaystyle \frac{(r-1)a_{r-1}}{2}\log(2\pi)-
				\left(\frac{r+1}{6}a_{r-1}-\frac{r-1}{2}\right)\log 2$ \\ 
				$G_2$
				& $\frac{5}{12}$
				& $\frac{5}{2}\log(2\pi)-\frac{1}{2}\log 2-\frac{1}{2}\log 3$ \\ 
				$F_4$
				& $\frac{385}{1152}$
				& $\frac{385}{48}\log(2\pi)-3\log 2-\frac{1}{3}\log 3$ \\ 
				$E_6$
				& $\frac{77}{324}$
				& $\frac{77}{9}\log(2\pi)-\frac{5}{6}\log 2-\frac{1}{18}\log 3$ \\ 
				$E_7$
				& $\frac{2431}{9216}$
				& $\frac{17017}{1024}\log(2\pi)-
				\frac{2841}{1024}\log 2-\frac{5}{12}\log 3$ \\ 
				$E_8$
				& $\frac{30808063}{99532800}$
				& $\frac{30808063}{829440}\log(2\pi)-\frac{188659}{18432}\log 2
				-\frac{1697}{648}\log 3-\frac{1}{5}\log 5$ \\ 
				\hline
			\end{tabular}
		}
		\caption{\small Values and derivatives at the origin of the normalized Witten zeta functions for all irreducible root systems, where $\displaystyle a_k:=\frac{k}{2^{2k-1}}\binom{2k}{k}$.}
		\label{table:Witten-zeta-values-at-zero}
	\end{table}

	The two theorems above arise from a more general class of functions, which we call \textit{complete conical zeta functions}. Let $\L:= (L_1,\ldots,L_N)$ be a multiset\footnote{i.e., $L_1,\ldots,L_N$ need not be pairwise distinct.} of nonzero rational linear forms on $\mathbb{R}^r$. When $L_1,\ldots,L_N$ span the dual space, we define the complete conical zeta function by
	$$\xi_\L(s) := \sum_{\boldsymbol{x} \in \mathbb{Z}^r}^{} {'}\frac{1}{|L_1(\boldsymbol{x})\cdots L_N(\boldsymbol{x})|^s},\qquad \Re(s)\gg 0,$$
	where the prime indicates that terms for which at least one $L_i(\boldsymbol{x})$ vanishes are omitted. It was briefly considered by Zagier \cite{zagier1994values}, who proved that $\xi_\L(2n)\in\mathbb{Q}\pi^{2Nn}$ for $n\in\mathbb{N}$. Such special values were later studied systematically in the framework of hyperplane arrangements \cite{szenes1998iterated,brion2000arrangement,Komori2014lattice,kadota2022parity,komori2023theory}.
	
	Let $\A:=\{\ker L_i:1\leq i\leq N\}$ be the underlying hyperplane arrangement of $\L$. The Witten zeta function $\xi_\Phi(s)$ is recovered when $\A$ is the Weyl arrangement of $\Phi$. We will prove that $\xi_\L(s)$ admits a meromorphic continuation to $\mathbb{C}$ and is analytic at every nonpositive integer.
	
	Let $L(\A)$ denote the intersection poset (under reverse inclusion) of $\A$, and let $\mu_\A(X,Y)$ be its M\"obius function; these notions are reviewed in Section 2.1.
	
	\begin{theorem}\label{thm:conical_0}
		The value of $\xi_\L(s)$ at the origin is
		$$\xi_\L(0) = \mu_\A(\mathbb{R}^r,\{0\}),$$
		i.e., the value of the M\"obius function between the minimal and maximal elements of the intersection poset.
	\end{theorem} 
	
	This result (Theorem \ref{values_at_nonpositive_even}) relates an analytically defined quantity to a combinatorial invariant of the arrangement $\A$. Since Weyl arrangements are free, the above result and Terao's factorization theorem \cite{terao1981generalized} immediately yield Theorem \ref{thm:witten_0}. 
	
	The derivative contains additional geometric information: unlike Theorem \ref{thm:conical_0}, which involves only the M\"obius invariant of the zero-dimensional flat $\{0\}$, its formula also receives contributions from one-dimensional flats.
	\begin{theorem}\label{thm:conical_0_der}
		For each one-dimensional flat $X\in L(\A)$, choose a generator $u_X \in \mathbb{Q}^r$, defined up to sign, such that $\mathbb{Z} u_X = X\cap\mathbb Z^r$, and put
		$$m_X:=\prod_{\substack{1\leq k\leq N\\ L_k(u_X)\neq 0}}|L_k(u_X)|\in\mathbb Q_{>0}.$$
		Then
		$$\xi_\L'(0)=N\mu_\A(\mathbb{R}^r,\{0\})\log(2\pi)+\sum_{\substack{X\in L(\A)\\\dim X=1}}\mu_\A(\mathbb{R}^r,X)\log m_X.$$
	\end{theorem}
	
	Theorem \ref{thm:witten_0_der} for Witten zeta functions follows readily from this result. Apart from the combinatorial quantities $\mu_\A(\mathbb{R}^r,X)$, the formula also contains arithmetic data $\log m_X$ attached to one-dimensional flats. The arithmetic nature of these latter terms is clarified by the following observation.
	
	\begin{corollary} \label{cor:intro-bad-primes}
		Assume the linear forms $L_1,\ldots,L_N$ have integral coefficients. Let $A$ be the $N\times r$ integer matrix whose rows are their coefficient vectors. We say that a prime $p$ is bad if $p$ divides one of its nonzero $r\times r$ minors. Then
		$$\xi_\L'(0) \in \mathbb{Z} \log(2\pi) + \bigoplus_{p \text{ bad}} \mathbb{Z}\log p.$$
	\end{corollary}
	
	For Weyl arrangements, the bad primes in this sense are precisely the primes appearing in \eqref{bad_primes_Weyl}. They coincide with the primes traditionally called bad for the root system \cite[Chapter~1.4]{Springer1970}. Equivalently, they are the prime divisors of the coefficients of the highest root when it is expressed in terms of the simple roots; see \cite{au2024vanishing}.
	
	The proofs of Theorems \ref{thm:conical_0} and \ref{thm:conical_0_der} have two complementary components, one analytic and one combinatorial. The analytic component begins with the Mellin-transform representation
	$$\xi_\L(s) = \frac{1}{\Gamma(s)^N} \int_{(0,\infty)^N} G_\L(\y) \y^{s-1} d\y,\quad \y^{s-1} := \prod_{i=1}^N y_i^{s-1}.$$
	Unlike the integrands associated with the Shintani-type zeta functions, $G_\L(\y)$ is generally a complicated exponential rational function, and no convenient expression is available for arbitrary $\L$. We therefore develop a framework for a general function $G(\y)$ satisfying the local regularity and Schwartz-type decay hypotheses formulated in Section~3. For the function
	$$I_G(s):=\dfrac{1}{\Gamma(s)^N}
	\int_{(0,\infty)^N}G(\y)\y^{s-1}\,d\y$$
	our main analytic result, Proposition \ref{prop:I-values}, gives formulas for $I_G(-n)$ and $I_G'(-n)$, where $n\in\mathbb{Z}_{\geq0}$, in terms of residues and Hadamard finite-part integrals. For the relation between Mellin transforms and Hadamard finite-part integrals, see \cite{monegato1998euler}.
	
	These formulas also apply to Shintani-type zeta functions, whose analytic continuation and special values at nonpositive integers have been studied extensively \cite{shintani1976evaluation,bruna2025polynomials,rutard2026values,eie1993special}. As a consequence of our formula for $I_G(-n)$, we obtain the following novel parity-vanishing criterion (Theorem \ref{thm:parity-vanishing}) for a general class of such functions.
	\begin{theorem} \label{thm:intro-parity-vanishing}
		Let $N\geq r$, and let $L_1,\ldots,L_{N}$ be nonzero rational linear forms with nonnegative coefficients such that $L_i(\boldsymbol{x})=x_i$ for $1\le i\le r$. Consider the function
		$$Z(s):=\sum_{n_1,\ldots,n_r\in\mathbb{N}}\frac{1}{\left(L_{1}(\boldsymbol{n})\cdots L_{N}(\boldsymbol{n})\right)^s}=\sum_{n_1,\ldots,n_r\in\mathbb{N}}\frac{1}{\left(n_1\cdots n_rL_{r+1}(\boldsymbol{n})\cdots L_{N}(\boldsymbol{n})\right)^s},\qquad \Re(s)\gg0.$$
		If $n\in\mathbb{N}$ and $Nn+r\equiv1\pmod{2}$, then $Z(-n)=0$.
	\end{theorem}
    
	For a general $G(\y)$, the formulas for $I_G(-n)$ and $I_G'(-n)$ do not simplify further. In the case of $\xi_\L(s)$, however, the combinatorics of the underlying hyperplane arrangement becomes decisive. The crucial feature is that $\xi_\L(s)$ sums over all lattice points in $\mathbb{Z}^r$, rather than over $\mathbb{N}^r$ as in Shintani-type zeta functions. This produces cancellations that greatly simplify the formulas. To organize these cancellations, we use valuations on polyhedral cones as a bookkeeping device for manipulating the otherwise unwieldy $G_\L(\y)$ \cite{BoussicaultFerayLascouxReiner2012,barvinok1992computing}. This interplay between the analytic and combinatorial arguments is one of the main themes of the article.
	
	As a further consequence of the cone-valuation technique, we obtain the following identity associated with a root system; see Proposition \ref{group_alg_identity}.
	\begin{proposition} \label{intro-group-alg-identity}
		Let $\lambda_1,\ldots,\lambda_r$ be the fundamental weights of $\Phi$. Then the following identity holds in the fraction field of the group algebra of the weight lattice:
		$$\sum_{w\in W}\prod_{i=1}^r\frac{1}{1-\e^{w\lambda_i}}=E(\Phi).$$
	\end{proposition}
	We shall see that this identity is closely related to the evaluation of $\zeta_\Phi(0)$ in Theorem \ref{thm:witten_0}. It is reminiscent of the better-known identity involving the simple roots \cite{macdonald1972poincare}:
	$$\sum_{w\in W}\prod_{\alpha \in \Delta}\frac{1}{1-\e^{w\alpha}}=1.$$
	
	The formulas for $\xi_\L(0)$ and $\xi_\L'(0)$ above suggest several possible extensions. One natural direction, not pursued in this article, is to introduce a rational translation and study
	$$\sum_{\boldsymbol{x} \in \mathbb{Z}^r}^{} {'}\frac{1}{|L_1(\boldsymbol{x} + \boldsymbol{c})\cdots L_N(\boldsymbol{x} + \boldsymbol{c})|^s},\qquad \boldsymbol{c}\in \mathbb{Q}^r.$$
	Such translated sums arise naturally in the study of $\zeta_G(s)$ for compact Lie groups $G$ that are \textit{not} simply-connected.
	
	The paper is organized as follows. Section~2 reviews the necessary background on hyperplane arrangements, valuations on polyhedral cones, and Weyl arrangements. Section~3 develops the analytic input described above. Sections~4 and~5 prove the two main formulas, Theorems \ref{thm:conical_0} and \ref{thm:conical_0_der}, while Section~6 specializes these results to Witten zeta functions and evaluates the quantities $m_\Phi(d)$.
	
	\section*{Acknowledgment}
    The authors would like to thank Prof. Kohji Matsumoto
for bringing several relevant references to their attention.
	The first author has received funding from the European Research Council (ERC) under the European Union’s Horizon 2020 research and innovation programme (grant agreement No. 101001179). 
	\section*{Declaration on AI usage}
	The authors used ChatGPT 5.6 Sol for auxiliary tasks while preparing the manuscript: checking the language, assisting with literature searches, typesetting Tables \ref{table:Witten-zeta-values-at-zero} and \ref{table:m-Phi-values}, generating the concrete numerical illustrations in Examples \ref{Ex_1} and \ref{Ex_2}, and producing the accompanying SageMath code. All generated material has been checked and revised by the authors. ChatGPT also assisted in exploring the proofs of Lemma \ref{lem:first-two-coefficients} and Proposition \ref{laurent_partial_residue}. The core mathematical ideas and overall strategy were developed by the authors, who take full responsibility for the content of the manuscript.
	
	\section{Preliminaries}
	\subsection{Hyperplane arrangements}
	Here we quickly recall the basic facts from the theory of hyperplane arrangement that we shall need. Readers can consult \cite[p.~389-496]{stanley2007geometric} for more background. \par
	Let $V$ be an $n$-dimensional real vector space, a \textit{hyperplane arrangement} $\mathcal{A}$ is a finite collection of (affine) hyperplanes in $V$. It is called \textit{central} if $0 \in \bigcap_{H\in \mathcal{A}} H$. It is called \textit{essential} if its normal vectors span $V$. Its \textit{intersection lattice} $L(\mathcal{A})$ is the collection of all intersections of hyperplanes in $\mathcal{A}$, namely:
	$L(\mathcal{A}) := \{H_1 \cap H_2 \cdots \cap H_k \neq \varnothing : H_i \in \mathcal{A}\}$, where for $k=0$ the intersection is understood to be $V$. Elements of $L(\mathcal{A})$ are called \textit{flats} of $\mathcal{A}$. Connected components of $V - \bigcup_{H\in \mathcal{A}} H$ are called \textit{chambers}, the set of chambers is denoted by $\Ch(\mathcal{A})$. We partially order $L(\mathcal{A})$ with respect to reverse inclusion $X \leq Y \iff X\supset Y$. Then $L(\mathcal{A})$ has a unique minimal element, which is $V$. \par
	Recall the notion of M\"obius function on a finite poset, which we denote as $\mu_\mathcal{A}(X,Y)$ for $X,Y\in L(\mathcal{A})$ and $X\leq Y$. The \textit{characteristic polynomial} of $\mathcal{A}$ is defined as
	$$\chi_\mathcal{A}(t) := \sum_{X\in L(\mathcal{A})} \mu_\A(V,X) t^{\dim X}.$$
	One can show that $\mu_\A(X,Y) (-1)^{\dim(X)-\dim(Y)} > 0$. Recall the \textit{localization} of $\A$ at $X\in L(\mathcal{A})$ is
	$$\A_X:=\{H\in\A:X\subset H\},$$
	which is a subarrangement of $\A$. Write $\ind_S$ for the indicator function of a subset $S \subset V$.
	
	\begin{proposition}\label{indicator_identity}
		For every hyperplane arrangement $\A$ in $V$, we have the following identity of functions on $V$:
		$$\sum_{D\in\Ch(\A)}\ind_{D}=\sum_{Y\in L(\A)} \mu_{\A}(V,Y)\,\ind_Y.$$
	\end{proposition}
	\begin{proof}
		Evidently,
		$$\sum_{D\in\Ch(\A)}\ind_D=\ind_{V-\bigcup_{H\in\A}H}.$$
		Fix $x\in V$, and define $X$ to be smallest flat containing $x$, i.e.
		$$X:=\bigcap_{\substack{H\in\A\\x\in H}}H,$$
		where the intersection of the empty family is understood to be $V$.
		For every $Y\in L(\A)$, we have $x\in Y\iff X\subset Y\iff Y\leq X$. Therefore
		$$\begin{aligned}
			\sum_{Y\in L(\A)}\mu_\A(V,Y)\ind_Y(x)
			&=\sum_{\substack{Y\in L(\A)\\Y\leq X}}\mu_\A(V,Y)=\begin{cases}
				1,&X=V,\\
				0,&X\neq V,
			\end{cases}
		\end{aligned}$$
		by the defining recurrence for the M\"obius function. Finally,
		$X=V$ if and only if $x\notin\bigcup_{H\in\A}H$. Thus the last
		display is precisely the value at $x$ of
		$\ind_{V-\bigcup_{H\in\A}H}$, completing the proof.
	\end{proof}
	
	\begin{lemma}\label{Weisner_lemma}
		Let $\A$ be a central and essential hyperplane arrangement in $V$. Then for any hyperplane $H\in \A$, we have
		$$\sum_{\substack{\dim X=1\\X\not\subset H}}\mu_{\mathcal A}(V,X)=-\mu_{\mathcal A}(V,\{0\}).$$
	\end{lemma}
	\begin{proof}
		Although a direct proof can be given, we deduce it as a consequence of Weisner’s theorem in the context of finite lattice \cite[Corollary~3.9.3]{stanley1986enumerative}. More precisely, let $L$ be a finite lattice with minimal element $\hat{0}$ and maximal element $\hat{1}$. Let $\hat{0}\neq a\in L$, then
		$$\sum_{x: x\vee a = \hat{1}} \mu(\hat{0},x) = 0.$$
		One specializes this assertion to $L(\A)$, which is a lattice with $\hat{0}=V$ and $\hat{1}=\{0\}$: centrality gives the lattice structure, while essentiality gives $\hat{1}=\{0\}$. Take $a = H$, there are two types of elements in $L(\A)$ with $X\cap H = x\vee a = \hat{1} = \{0\}$: either $X = \{0\}$ or $\dim X = 1$ and $X\not\subset H$, this gives the claim immediately.
	\end{proof}
	
	\subsection{Valuation on polyhedral cones}
	Let $V = \mathbb{R}^r$, fix a full-rank lattice $\Lambda$ in $V$. Let $\mathbb{Z}[\Lambda]$ be the group algebra with basis $\{\e^x\}_{x\in\Lambda}$ and multiplication $\e^x \e^y = \e^{x+y}$. If $\{v_1,\ldots,v_r\}$ is a basis of $\Lambda$, then
	$$\mathbb{Z}[\Lambda] \cong \mathbb{Z}[\e^{\pm v_1},\ldots,\e^{\pm v_r}].$$
	Write $\mathbb{Q}(\Lambda) := \Frac(\mathbb{Z}[\Lambda])$. Also let $\mathbb{Z}\{\{\Lambda\}\}$ be the abelian group whose elements are of the form $\sum_{x\in \Lambda} c_x \e^{x}$ $(c_x\in \mathbb{Z})$ with no finiteness condition on the support. Note that $\mathbb{Z}\{\{\Lambda\}\}$ is not a ring (since the obvious multiplication can involve infinitely many operations), but it contains $\mathbb{Z}[\Lambda]$ and is a $\mathbb{Z}[\Lambda]$-module.
	For an element $h\in \mathbb{Z}\{\{\Lambda\}\}$, if there are $p,q\in \mathbb{Z}[\Lambda]$ with $q\ne 0$ such that $qh=p$ in $\mathbb{Z}\{\{\Lambda\}\}$, then we say $h$ \emph{sums to} $\frac{p}{q}\in\mathbb{Q}(\Lambda)$ and interpret $h$ as $\frac{p}{q}\in\mathbb{Q}(\Lambda)$. This value is well-defined.\par
	For $l\in V_{\mathbb{C}}^*:=\operatorname{Hom}(V,\mathbb{C})$, write $\langle l,x\rangle :=l(x)$ for the natural paring between $V_{\mathbb{C}}^*$ and $V$. Define the ring homomorphism $\operatorname{ev}_l:\mathbb{Z}[\Lambda]\to \mathbb{C}$, $\e^x \mapsto e^{-\langle l, x \rangle}$, here $e$ is the base of natural logarithm. If $h\in \mathbb{Z}\{\{\Lambda\}\}$ sums to $\frac{p}{q}\in \mathbb{Q}(\Lambda)$ with $\operatorname{ev}_l(q)\ne 0$, then we write $h(l):=\operatorname{ev}_l(p)/\operatorname{ev}_l(q)$. This is independent of the chosen representation.\par
	A \textit{(rational polyhedral) cone} $K$ is a set of the form $\mathbb{R}_{\geq 0} u_1 + \cdots + \mathbb{R}_{\geq 0}u_M$ for some vectors $u_i\in \Lambda$. We define its \textit{dual cone} as 
	$$K^\ast = \{ l\in V^\ast : \langle l,x\rangle \geq 0 \text{ for all }x\in K\}.$$\par
	For any subset $S\subset V$, define
	$$\mathrm{H}_S := \sum_{x\in S\cap \Lambda} \e^{x} \in \mathbb{Z}\{\{\Lambda\}\}.$$
	For a cone $K$, $K\cap \Lambda$ is an affine semigroup, and $\mathrm{H}_K$ is its finely graded \textit{Hilbert series}.
	It enjoys the following properties \cite[Proposition~2.2]{BoussicaultFerayLascouxReiner2012}.
	
	\begin{proposition}\label{prop:Hilbert-series-properties}
		(a) $\mathrm{H}_K$ sums to an element in $\mathbb{Q}(\Lambda)$, so we can interpret $\mathrm{H}_K \in \mathbb{Q}(\Lambda)$.\\
		(b) If $l$ is in the interior of $K^\ast$, then the sum $\sum_{x\in K\cap \Lambda} e^{-l(x)}$ converges absolutely as an infinite series, and its value equals $\mathrm{H}_K(l)$.\\
		(c) Let $K$ be a simplicial cone, and let $\boldsymbol{u}=\{u_1,\ldots,u_d\}$ be the set of primitive generators of its extreme rays\footnote{See the reference above for definitions of these terminologies.}. Put $\Pi_{\boldsymbol{u}}:=\left\{\sum_{i=1}^d c_i u_i : 0\le c_i<1\right\}.$
		Then $$\mathrm{H}_K=\frac{\sum_{u\in \Pi_{\boldsymbol{u}}\cap \Lambda} \e^{u}}{\prod_{i=1}^d (1-\e^{u_i})}.$$
		(d) (valuation criterion) If there is a linear relation $\sum_{i=1}^t a_i \ind_{K_i} = 0$ between indicator functions of rational cones $K_i$, then $\sum_{i=1}^t a_i \mathrm{H}_{K_i} = 0$.
	\end{proposition}
	For (b), the cited proposition assumes that $K$ is full-dimensional. The formulation above follows by applying it in $\operatorname{Span}_{\mathbb{R}}K$ with the full-rank lattice $\Lambda\cap\operatorname{Span}_{\mathbb{R}}K$. 
	
	As an example, if $K = \mathbb{R}_{\geq 0} v_1 + \cdots + \mathbb{R}_{\geq 0}v_r$ with $\{v_1,\ldots,v_r\}$ a basis of $\Lambda$, then $$\mathrm{H}_K = \sum_{m_1,\ldots,m_r\geq 0} \prod_{i=1}^r \e^{m_i v_i} = \prod_{i=1}^r (1-\e^{v_i})^{-1},
	\qquad \mathrm{H}_K(l) = \prod_{i=1}^r (1-e^{-l(v_i)})^{-1}.$$\par 
	A \textit{half-open cone} is a subset of $V$ of the form
	$$k_1 u_1 + \cdots + k_M u_M, \qquad u_i\in \Lambda, \ k_i\in\{ \mathbb{R}_{>0}, \mathbb{R}_{\geq 0}\}.$$
	Using the inclusion-exclusion principle, property (a) extends to finite unions of half-open cones and property (d) holds for such sets as well. Property (b) also holds for a half-open cone $K$, with $K^*$ replaced by $(\overline{K})^*$. 
	By slight abuse of terminology, we also call $\mathrm{H}_S$ the Hilbert series of $S$ when $S$ is a finite union of half-open rational cones.
	
	\begin{lemma}\label{pointed_vanishing_lemma}
		If a half-open cone $K$ is invariant under translation by a nonzero rational vector, then $\mathrm{H}_K$ sums to $0$.
	\end{lemma}
	\begin{proof}
		Suppose $K+w=K$ for a nonzero rational vector $w$, we can assume $w\in \Lambda$. Then
		$$\mathrm H_{K} = \sum_{x\in K\cap \Lambda }\e^{x} =\sum_{x\in K\cap \Lambda }\e^{x+w}=\e^{w}\mathrm H_{K}.$$
		Thus $(1-\e^w)\mathrm{H}_K=0$, so $\mathrm{H}_{K}$ sums to $0$.
	\end{proof}
	
	Now let $\mathcal{A}$ be a central and essential rational\footnote{this means each hyperplane in $\mathcal{A}$ is defined by the kernel of an element in $\operatorname{Hom}(\mathbb{Q}\otimes_\mathbb{Z} \Lambda,\mathbb{Q})$. If $\Lambda  = \mathbb{Z}^r$, this means the hyperplanes are defined by linear forms with rational coefficients.} hyperplane arrangement in $V$. For each chamber $D\in \Ch(\A)$, both $D$ and its closure $\overline{D}$ are half-open cones. Thus their Hilbert series are defined and belong to $\mathbb{Q}(\Lambda)$. 
	\begin{lemma}\label{hilbert_chamber_identity}
		For a central and essential rational arrangement $\mathcal{A}$, we have $$\sum_{D\in \Ch(\A)} \mathrm{H}_D = \mu_\A(V,\{0\}).$$
	\end{lemma}
	\begin{proof}
		Apply the Hilbert-series valuation to the indicator-function identities in Proposition~\ref{indicator_identity}. This gives, in $\mathbb Q(\Lambda)$,
		$$\sum_{D\in\Ch(\A)}\mathrm H_D
		=\sum_{X\in L(\A)}\mu_\A(V,X)\mathrm H_X.$$
		Because $\A$ is central, every flat $X$ is a linear subspace. If $\dim X>0$, rationality of $X$ provides a nonzero vector $w\in X\cap\Lambda$, and $X+w=X$. Lemma~\ref{pointed_vanishing_lemma} therefore gives $\mathrm H_X=0$. Thus only $X$ with $\dim X =0$ survives. Since $\A$ is essential, this means $X = \{0\}$, and $\mathrm H_{\{0\}}=\e^0=1$. 
	\end{proof}
	
	\subsection{Root system and Weyl arrangement}\label{root_system_subsection}
	For a crystallographic root system $\Phi$ of rank $r$ in a $r$-dimensional Euclidean space $V$, with inner product pairing $(\cdot,\cdot)$ and coroot $\alpha^\vee := \frac{2\alpha}{(\alpha,\alpha)}$. We shall adopt the following notations associated to $\Phi$:
	\begin{itemize}[leftmargin=*]
		\item $\Delta := \{\alpha_1,\ldots,\alpha_r\}$: a fixed set of simple roots, with $\{\lambda_1,\ldots,\lambda_r\}$ the corresponding fundamental weights;
		\item $\Phi^+$: positive roots with respect to $\Delta$; 
		\item $N := |\Phi^+|$, the number of positive roots; 
		\item $W$: Weyl group of $\Phi$;
		\item $E(\Phi)$: the product of exponents $e_1,\ldots,e_r$ of $W$;
		\item $\Phi_J, W_J$: the parabolic root system and Weyl group for a subset $J\subset \Delta$;
		\item $H_\gamma = \{x\in V: (x,\gamma) = 0\}$ : the reflecting hyperplane of $\gamma\in \Phi$;
		\item $\mathcal{A}$ : the essential and central hyperplane arrangement (known as \textit{Weyl arrangement}) formed by $\{H_\gamma : \gamma\in \Phi\}$.
	\end{itemize}
	
	There is an extensive literature on Weyl arrangements; see, for example, \cite{athanasiadis2000deformations,yoshinaga2014freeness,douglass2012invariants}. We cite the following result.
	
	\begin{lemma}\label{prod_exp}
		The characteristic polynomial of $\mathcal{A}$ is
		$$\chi_{\mathcal{A}}(t) = \prod_{i=1}^r (t-e_i).$$
		In particular, $\mu_\A(V,\{0\}) = (-1)^r E(\Phi)$.
	\end{lemma}
	\begin{proof}
	It is known that the Weyl arrangement is free \cite[Theorem~6.60]{orlik1992arrangements}, with exponents equal to the exponents $e_1,\ldots,e_r$ of the Weyl group. Terao's factorization theorem \cite{terao1981generalized} then gives the result. See also \cite[Corollary~2.6]{athanasiadis2000deformations}.
	\end{proof}
	
	Combining the valuation technique above with Lemma \ref{prod_exp} yields the following identity stated in the introduction. Although our approach in Section 4 uses the more general identity in Lemma \ref{hilbert_chamber_identity},
the following identity arises naturally in the evaluation of $\zeta_\Phi(0)$.
	\begin{proposition}\label{group_alg_identity}
		The following identity holds in the fraction field of the group algebra of the weight lattice:
		$$\sum_{w\in W}\prod_{i=1}^r\frac{1}{1-\e^{w\lambda_i}}=E(\Phi).$$
	\end{proposition}
	\begin{proof}
		Take $\Lambda$ in Section~2.2 to be the weight lattice $\Lambda=\bigoplus_{i=1}^r\mathbb Z\lambda_i,$
		and let $C=\sum_{i=1}^r\mathbb R_{>0}\lambda_i$ be the open fundamental Weyl chamber. The Weyl group acts simply transitively on the chambers of $\A$, so they are precisely the open cones $wC$, with $w\in W$. We also have
		$$\Lambda\cap wC
		=\left\{\sum_{i=1}^r m_iw\lambda_i:m_1,\ldots,m_r\in\mathbb{N}\right\}.$$
		Consequently, in $\mathbb Q(\Lambda)$,
		$$\mathrm H_{wC}
		=\sum_{m_1,\ldots,m_r\geq1}\e^{\sum_i m_iw\lambda_i}
		=\prod_{i=1}^r\frac{\e^{w\lambda_i}}{1-\e^{w\lambda_i}}
		=(-1)^r\prod_{i=1}^r \frac{1}{1-\e^{-w\lambda_i}}.$$
		Lemma \ref{hilbert_chamber_identity} now yields
		$$\sum_{w\in W}\prod_{i=1}^r\frac{1}{1-\e^{-w\lambda_i}}
		=(-1)^r \sum_{D\in\Ch(\A)}\mathrm H_{D}
		=(-1)^r \mu_\A(V,\{0\}).$$
		Lemma \ref{prod_exp} says $\mu_\A(V,\{0\}) =(-1)^r E(\Phi)$, so $$\sum_{w\in W}\prod_{i=1}^r\frac{1}{1-\e^{-w\lambda_i}} = E(\Phi).$$
		Replacing $\lambda_i$ by $-\lambda_i$ (corresponding to an isomorphism in $\mathbb{Q}(\Lambda)$) completes the proof.
	\end{proof}
	
	We show the facts concerning the one-dimensional flats of Weyl arrangements that will be used in Section 6.
	\begin{lemma}\label{para_exp}
		Let $J = \Delta - \{\alpha_i\}$ for some $1\leq i \leq r$. We have
		$$\mu_{\mathcal{A}}(V, \mathbb{R}\lambda_i)=(-1)^{r-1} E(\Phi_J).$$
	\end{lemma}
	\begin{proof}
		Let $X = \mathbb{R}\lambda_i$, $\A_X$ be the localized arrangement: $\A_X=\{H_\gamma : \gamma\in\Phi_J\}$. Let $V_J$ be the $\mathbb{R}$-span of $\Delta - \{\alpha_i\}$. Also let $\mathcal{B} = \{H_\gamma\cap V_J : \gamma\in\Phi_J\}$, which is a hyperplane arrangement of $V_J$. Then it is easy to see
		$$L(\mathcal{A}_X) \to L(\mathcal{B}),\qquad Y \mapsto Y\cap V_J$$
		is a bijection and it preserves the partial order on $L(\A_X)$ and $L(\mathcal{B})$, \textit{viz.} it is a poset isomorphism. Thus they have the same M\"obius function: $\mu_\A(V,X) = \mu_\mathcal{B}(V_J, \{0\})$. Also, $\mathcal{B}$ is the arrangement associated to root system $\Phi_J \subset V_J$. Lemma \ref{prod_exp} applied to $\Phi_J$ completes the proof.
	\end{proof}
	
	\begin{lemma}\label{aux_lemma}
		(a) Every one-dimensional flat of $\A$ is of the form $$\mathbb{R}w\lambda_i, \qquad w\in W, \ 1\leq i\leq r.$$
		(b) The stabilizer of $\lambda_i$ in $W$ is the parabolic subgroup $W_{\Delta-\{\alpha_i\}}$.
	\end{lemma}
	\begin{proof}
		(a) Consider the set $C = \mathbb{R}_{>0} \lambda_1 + \cdots + \mathbb{R}_{>0} \lambda_r$. Since Weyl chambers are fundamental domains of $W$ acting on $V$, we have
		$$\bigsqcup_{w\in W} wC = V - \bigcup_{\gamma\in \Phi^+} H_\gamma.$$
		Hence the one-dimensional flats, as intersection of some of these $H_\gamma$, correspond to boundary lines of $\overline{wC}$.\\
		(b) According to \cite[Theorem~1.12]{humphreys1992reflection}, the stabilizer of $\lambda_i$ in $W$ is generated by the reflection $$s_\gamma(x) = x - (x,\gamma^\vee)\gamma,\qquad \gamma\in \Phi$$ contained it. This is exactly $W_{\Delta-\{\alpha_i\}}$.
	\end{proof}
	
	\section{Multiple Mellin transforms at nonpositive integers}
	
	This section develops the analytic input used in the subsequent study of conical zeta functions. Although our main applications concern the origin, the general Mellin-transform argument naturally treats all nonpositive integers at once. We therefore derive formulas for the values and first derivatives of a class of normalized multiple Mellin transforms at $s=-n$ ($n\in\mathbb Z_{\geq0}$).
	
	\subsection{Admissible functions}
	Let $N\in\mathbb{N}$ and let $G:(0,\infty)^N\to\mathbb{C}$.
	We write $G\in\mathcal{R}_N$ if there exist real linear forms $l_1,\ldots,l_k$
	non-vanishing on $(0,\infty)^N$ such that
	$H:=(l_1\cdots l_k)G$ satisfies the following conditions:
	\begin{enumerate}
		\item[(A1)] $H$ extends holomorphically to a complex neighborhood of $[0,\infty)^N$.
		\item[(A2)] For every $A\ge0$ and every $\boldsymbol{\alpha}=(\alpha_1,\ldots,\alpha_N)\in\mathbb{Z}_{\ge0}^N$,
		\[
		\sup_{\y\in[0,\infty)^N}
		(1+\|\y\|)^A\bigl|\partial_{\y}^{\boldsymbol{\alpha}} H(\y)\bigr|<\infty,
		\]
	\end{enumerate}
	where $\|\y\|$ is the Euclidean norm of $\y$ and $\partial_{\y}^{\boldsymbol{\alpha}}=\partial_{y_1}^{\alpha_1}\cdots \partial_{y_N}^{\alpha_N}$.
	
	\begin{example} \label{ex:R-functions}
		Let $\boldsymbol{x}\cdot\y$ denote the standard Euclidean inner product of $\boldsymbol{x}, \y \in \mathbb{R}^N$.
		Suppose that $G$ is a finite linear combination of functions of the form
		\[
		\dfrac{e^{-\boldsymbol{b}\cdot \y}P(e^{-\boldsymbol{c}_1\cdot \y},\ldots,e^{-\boldsymbol{c}_q\cdot \y})}
		{\prod_{\nu=1}^M(1-e^{-\boldsymbol{a}_\nu\cdot \y})},
		\]
		where
		$\boldsymbol{b}\in\mathbb{R}_{>0}^N$,
		$\boldsymbol{a}_\nu\in\mathbb{R}_{\ge0}^N\setminus\{0\}$,
		$\boldsymbol{c}_\mu\in\mathbb{R}_{\ge0}^N$,
		and $P$ is a polynomial.
		Then $G\in\mathcal{R}_N$.
		Indeed, we may choose $l_\nu(\y)=\boldsymbol{a}_\nu\cdot \y$ ($1\le \nu\le M$).
	\end{example}
	
	Set $X_N=(0,\infty)\times(0,1)^{N-1}$.
	For a function $h:X_N\to\mathbb C$, consider the following conditions:
	\begin{enumerate}
		\item[(B1)] $h$ extends holomorphically to a complex neighborhood of
		$[0,\infty)\times[0,1]^{N-1}$.
		\item[(B2)] For every $A\ge0$ and every
		$\boldsymbol{\beta}=(\beta_2,\ldots,\beta_N)\in\mathbb{Z}_{\ge0}^{N-1}$,
		\begin{equation}\label{eq:h-sigma-decay}
			\sup_{\substack{z_1\ge1\\0<z_2,\ldots,z_N<1}}
			z_1^A
			\left|\partial_{z_2}^{\beta_2}\cdots\partial_{z_N}^{\beta_N}h(\z)\right|<\infty.
		\end{equation}
	\end{enumerate}
	When $N=1$, \textup{(B2)} means simply
	$\sup_{z_1\ge1}z_1^A|h(z_1)|<\infty$ for every $A\ge0$.
	
	Define
	\begin{equation}\label{eq:pi-def}
		\pi(\z)=(y_1,\ldots,y_N)=(z_1,z_1z_2,\ldots,z_1\cdots z_N).
	\end{equation}
	For $\sigma\in \mathfrak S_N$, the symmetric group on $N$ elements, and any $N$-variable function $F$,
	we define
	\begin{equation}\label{eq:sym-group-action}F^\sigma(y_1,\ldots,y_N):=F(y_{\sigma(1)},\ldots,y_{\sigma(N)}).\end{equation}
	
	\begin{lemma}\label{lem:G-pullback}
		Let $G\in\mathcal{R}_N$. For every $\sigma\in \mathfrak S_N$,
		there exists $\m_\sigma=(m_{\sigma,1},\ldots,m_{\sigma,N})\in\mathbb{Z}_{\ge0}^N$ such that
		\[
		h_\sigma:X_N\to\mathbb{C},
		\qquad h_\sigma(\z):=\z^{\m_\sigma}G^\sigma(\pi(\z))
		\]
		satisfies \textup{(B1)} and \textup{(B2)}, where $\z^{\m_\sigma}=\prod_{i=1}^N z_i^{m_{\sigma,i}}$.
	\end{lemma}
	
	\begin{proof}
		In the case $N=1$ the assertion immediately follows from \textup{(A1)} and \textup{(A2)}.
		
		Assume that $N\ge 2$.
		Fix $\sigma\in\mathfrak S_N$.
		Replacing any $l_j$ by $-l_j$ if necessary,
		we may assume that all coefficients of every $l_j$ are nonnegative.
		Write
		\begin{equation*}
			l_j^\sigma(\y)
			=\sum_{k=1}^N a_{j,k}y_k,
			\qquad a_{j,k}\ge0,
		\end{equation*}
		and let $r_j$ be the least index such that $a_{j,r_j}>0$.
		Then
		\[
		l_j^\sigma(\pi(\z))
		=z_1\cdots z_{r_j}U_j(z_{r_j+1},\ldots,z_N),
		\]
		where
		\[
		U_j=U_j(z_{r_j+1},\ldots,z_N)
		:=a_{j,r_j}+a_{j,r_j+1}z_{r_j+1}+\cdots+a_{j,N}z_{r_j+1}\cdots z_N.
		\]
		Set
		$m_{\sigma,i}:=\#\{j:r_j\ge i\}$ ($1\le i\le N$).
		Since $H^\sigma=(l_1^\sigma\cdots l_k^\sigma)G^\sigma$,
		we have
		\begin{equation}\label{eq:h-factorization}
			h_\sigma(\z)
			=\dfrac{H^\sigma(\pi(\z))}{\prod_{j=1}^kU_j}.
		\end{equation}
		
		By \textup{(A1)}, $H^\sigma$ is holomorphic on a complex neighborhood $U$ of $[0,\infty)^N$,
		so $H^\sigma\circ\pi$ is holomorphic on $\pi^{-1}(U)$,
		which contains $[0,\infty)\times[0,1]^{N-1}$.
		Moreover, $U_j\ge a_{j,r_j}>0$ on $[0,1]^{N-1}$.
		Hence $\prod_jU_j$ is nonzero on a complex neighborhood of $[0,1]^{N-1}$,
		and \eqref{eq:h-factorization} proves \textup{(B1)}.
		
		For \textup{(B2)}, put $Q(\z)=\prod_{j=1}^kU_j^{-1}$.
		Since $Q$ is independent of $z_1$,
		all its derivatives in $z_2,\ldots,z_N$ are bounded on $[0,\infty)\times[0,1]^{N-1}$.
		By the chain rule, for every
		$\boldsymbol{\beta}\in\mathbb Z_{\ge0}^{N-1}$ there is a constant
		$C_{\boldsymbol{\beta}}>0$ such that,
		for $z_1\ge1$ and $0<z_2,\ldots,z_N<1$,
		\[
		\left|\partial_{z_2}^{\beta_2}\cdots\partial_{z_N}^{\beta_N}
		h_\sigma(\z)\right|
		\le
		C_{\boldsymbol{\beta}}
		\sum_{|\boldsymbol{\alpha}|\le|\boldsymbol{\beta}|}
		z_1^{|\boldsymbol{\alpha}|}
		\left|(\partial_{\y}^{\boldsymbol{\alpha}}H^\sigma)(\pi(\z))\right|,
		\]
		where $|\boldsymbol{\alpha}|=\alpha_1+\cdots+\alpha_N$ for $\boldsymbol{\alpha}=(\alpha_1,\ldots,\alpha_N)$.
		Since $z_1\le \|\pi(\z)\|$, condition \textup{(A2)} gives \eqref{eq:h-sigma-decay}.
	\end{proof}
	
	\subsection{Hadamard finite parts and Laurent coefficients}
	
	Let $R\in \mathbb{R}_{>0}\cup \{\infty\}$.
	Let $\rho:(0,R]\to\mathbb{C}$ be Lebesgue integrable on $[\delta,R]$
	for every $0<\delta<R$ and extend meromorphically to a complex neighborhood of $0$.
	Then the cutoff integral has an expansion of the form
	$$\int_\varepsilon^R\rho(x)\,dx
	=\sum_{q=1}^M c_q\varepsilon^{-q}+c_{\log}\log\varepsilon+C+O(\varepsilon)
	\qquad(\varepsilon\downarrow0).$$
	The constant $C$ is the Hadamard finite part \cite{Hadamard}, denoted by
	$$\FP\int_0^R\rho(x)\,dx.$$
	For example, for $\delta >0, n\in \mathbb{Z}$, we have
	\begin{equation}\label{FP_simple_ex}\FP \int_0^\delta x^n \, dx = \begin{cases}\frac{\delta^{1+n}}{1+n} & n\neq -1\\ \log \delta & n=-1\end{cases}.\end{equation}
	
	Let $\omega=f(\z)\,dz_1\cdots dz_N$ be an $N$-form on $X_N$.
	We call $\omega$ \emph{admissible}
	if there is $\m=(m_1,\ldots,m_N)\in\mathbb{Z}_{\ge0}^N$ such that
	$h(\z):=\z^{\m} f(\z)$ satisfies \textup{(B1)} and \textup{(B2)}.
	
	Fix $k_1,\ldots,k_N\in\mathbb{N}$.
	For an admissible $N$-form $\omega$ on $X_N$, define
	$$M_\omega(s):=\int_{X_N}\prod_{j=1}^Nz_j^{k_js}\,\omega,
	\qquad \Re(s)\gg 0.$$

	Let $\omega$ be an admissible $N$-form on $X_N$.
	Let $\Res{\z=0}\omega$ denote the coefficient of $z_1^{-1}\cdots z_N^{-1}$
	in the Laurent expansion of the scalar coefficient of $\omega$,
	and define partial residues similarly.
	For $1\le j\le N$, set
	$$\rho_j(z_j)\,dz_j:=\Res{z_i=0\ (i\ne j)}\omega,$$
	where the reside is taken with respect to $\z$-variables except $z_j$.
	Let $R_1=\infty$ and $R_j=1$ ($2\le j\le N$).
	Since $\omega$ is admissible,
	each $\rho_j$ is meromorphic at $0$ and integrable on $[\delta,R_j]$ for every $0<\delta<R_j$;
	when $j=1$, condition \textup{(B2)} also gives rapid decay at $+\infty$.
	Thus all finite parts below are well defined.
	
	The next lemma gives a meromorphic continuation of $M_\omega(s)$ and
	identifies its two leading Laurent coefficients at $s=0$ in terms of residues and Hadamard finite parts.
	Such relations in one variable are classical; see, for example, \cite{monegato1998euler} or Corollary \ref{corollary_HDF_1dim} below. 
	
	\begin{lemma}\label{lem:first-two-coefficients}
		The integral $M_\omega(s)$  is absolutely convergent for $\Re(s)\gg0$ and admits a meromorphic continuation to $\mathbb{C}$.
		Set $K_N:=\prod_{j=1}^Nk_j$.
		Then, at $s=0$,
		\begin{equation}\label{eq:first-two-coefficients}
			M_\omega(s)
			=\dfrac{\Res{\z=0}\omega}{K_Ns^N}
			+\dfrac{1}{s^{N-1}}\sum_{j=1}^N
			\dfrac{1}{\prod_{i\ne j}k_i}
			\FP\int_0^{R_j}\rho_j(z_j)\,dz_j
			+O(s^{-N+2}).
		\end{equation}
	\end{lemma}
	
	\begin{proof}
		Absolute convergence for $\Re(s)\gg0$ follows directly from admissibility.
		
		Fix $\delta\in(0,1)$ and put $[N]:=\{1,2,\ldots,N\}$.
		Up to sets of measure zero, decompose $X_N$ into
		$$X_T(\delta):=
		\left\{\z\in X_N:
		\begin{array}{ll}
			0<z_i<\delta& \text{if $i\in T$},\\
			\delta<z_i<R_i& \text{if $i\notin T$}
		\end{array}\right\},
		\qquad T\subset[N],$$
		and let $M_T(s)$ denote the corresponding integral. Then $$M_\omega(s) = \sum_{T\subset [N]} M_T(s).$$
		
		We first continue $M_T$ meromorphically.
		Let $K\subset\mathbb{C}$ be compact, and choose
		$d\in\mathbb R$ with $d<\min_{s\in K}\Re(s)$.
		For each $i\in T$, choose $M_i\ge1$ such that $k_i d+M_i>-1$.
		Since $f(\z)=\z^{-\m}h(\z)$, Taylor expansion of $h$ in $z_i$ gives
		$$f(\z)=\sum_{q=-m_i}^{M_i-1}c_{i,q}(\z_{-i})z_i^q
		+z_i^{M_i}r_i(\z;M_i),$$
		where $\z_{-i}:=(z_1,\ldots,z_{i-1},z_{i+1},\ldots,z_N)$.
		After multiplication by $\prod_{j\ne i}z_j^{m_j}$,
		the coefficient functions and remainder are holomorphic in the remaining variables.
		Hence the expansion can be repeated in the remaining variables of $T$.
		If $1\notin T$, only derivatives in $z_2,\ldots,z_N$ occur in this process,
		so these functions also inherit from \textup{(B2)} rapid decay as $z_1\to\infty$,
		uniformly in the bounded variables.
		
		Apply these expansions successively in all variables $z_i$ with $i\in T$ to every coefficient
		and remainder obtained at the preceding steps. Each resulting term has the form
		\begin{equation}
			\label{eq:iterated-Laurent-term}
			\left(\prod_{i\in U}z_i^{q_i}\right)
			\left(\prod_{j\in T\setminus U}z_j^{M_j}\right)
			g_{T,U,\boldsymbol{q}}
			\bigl((z_k)_{k\in[N]\setminus U}\bigr),
		\end{equation}
		where $U\subset T$ and $q_i\in\{-m_i,\ldots,M_i-1\}$ for $i\in U$.
		The function $g_{T,U,\boldsymbol{q}}$ inherits the similar properties above.
		After integrating the variables in $T$ when $\Re(s)\gg 0$,
		each term in \eqref{eq:iterated-Laurent-term}
		contributes
		$$\dfrac{F_{T,U,\boldsymbol{q}}
			\bigl(s;(z_k)_{k\in[N]\setminus T}\bigr)}
		{\displaystyle\prod_{i\in U}(k_i s+q_i+1)},$$
		with numerator holomorphic in $s$ for $\Re(s)>d$.
		The remaining variables range over compact intervals,
		except that $z_1\in(\delta,\infty)$ when $1\notin T$;
		in that case rapid decay gives absolute and locally uniform convergence of the $z_1$-integral.
		Hence
		\begin{equation}
			\label{eq:MT-expression}
			M_T(s)=
			\sum_{U\subset T}
			\sum_{\boldsymbol{q}}
			\dfrac{F_{T,U,\boldsymbol{q}}(s)}
			{\displaystyle\prod_{i\in U}(k_i s+q_i+1)}
		\end{equation}
		for $\Re(s)\gg0$,
		where the right-hand side is meromorphic on $\Re(s)>d$.
		Thus this continues $M_T$ meromorphically to a neighborhood of $K$,
		and hence to $\mathbb{C}$.
		Summing over $T$ gives the continuation of $M_\omega$.
		
		Formula \eqref{eq:MT-expression} also shows that a pole at $s=0$ can arise only from factors with $q_i=-1$. Thus the pole order of $M_T$ at $s=0$ is at most $|T|$.
		Consequently, up to $O(s^{-N+2})$,
		only $T=[N]$ and $T=[N]\setminus\{j\}$ ($j\in [N]$) can contribute.
		
		We now compute these contributions more explicitly.
		Shrinking $\delta$ if necessary, $f$ has a Laurent expansion
		$$f(\z)=\sum_{\boldsymbol{\nu}:\nu_i\ge-m_i}
		a_{\boldsymbol{\nu}}\z^{\boldsymbol{\nu}}
		\qquad(0<|z_i|<\delta).$$
		This expansion gives
		\begin{equation}
			\label{eq:corner-series}
			M_{[N]}(s)=\sum_{\nu_i\ge-m_i}
			a_{\boldsymbol{\nu}}\prod_{i=1}^N
			\frac{\delta^{k_i s+\nu_i+1}}{k_i s+\nu_i+1},\qquad \Re(s)\gg 0.
		\end{equation}
		This shows that only terms for which at least $N-1$ of the $\nu_i$ are $-1$
		can contribute to $s^{-N}$ or $s^{-N+1}$.
		For each $j$, write
		$$\rho_j(z_j)=\sum_{q\ge-m_j}a_{j,q}z_j^q,\qquad
		a_{j,q}:=a_{(-1,\ldots,-1,q,-1,\ldots,-1)},$$
		where the $j$-th entry is $q$ and all the other entries are $-1$.
		Then $a_{j,-1}=\Res{\z=0}\omega$, and \eqref{eq:corner-series} gives
		\begin{equation} \label{eq:corner-expansion}
			M_{[N]}(s)
			=\dfrac{\Res{\z=0}\omega}{K_Ns^N}+\dfrac1{s^{N-1}}
			\sum_{j=1}^N\dfrac1{\prod_{i\ne j}k_i}
			\left(\Res{\z=0}\omega \log\delta+
			\sum_{\substack{q\ge-m_j\\q\ne-1}}
			a_{j,q}\dfrac{\delta^{q+1}}{q+1}
			\right)+O(s^{-N+2}).
		\end{equation}
		From equation \eqref{FP_simple_ex}, the expression in parentheses is exactly
		$\FP\int_0^\delta\rho_j(z_j)\,dz_j$.
		
		Fix $j$.
		On $X_{[N]\setminus\{j\}}(\delta)$,
		the only contribution of pole order $N-1$ is obtained
		by selecting the $z_i^{-1}$-coefficient for every $i\ne j$.
		By definition, the resulting coefficient is $\rho_j(z_j)$.
		Therefore
		\begin{align}
			\label{eq:adjacent}
			M_{[N]\setminus\{j\}}(s)
			&=\left(\prod_{i\ne j}\dfrac{\delta^{k_i s}}{k_i s}\right)
			\int_\delta^{R_j}z_j^{k_j s}\rho_j(z_j)\,dz_j
			+O(s^{-N+2})\\
			&=\dfrac1{s^{N-1}\prod_{i\ne j}k_i}
			\int_\delta^{R_j}\rho_j(z_j)\,dz_j
			+O(s^{-N+2}).\notag
		\end{align}
		
		Finally, summing \eqref{eq:corner-expansion} and \eqref{eq:adjacent} and using
		$$\FP\int_0^{R_j}\rho_j
		=\FP\int_0^\delta\rho_j+\int_\delta^{R_j}\rho_j$$
		gives \eqref{eq:first-two-coefficients}.
	\end{proof}
	
	\begin{corollary}\label{corollary_HDF_1dim}
		Let $N=1$ and let $\omega = f(z)\,dz$ be an admissible $1$-form on $X_1=(0,\infty)$. Then $$M_\omega(s) = \int_0^\infty f(z) z^s dz$$
		extends memormorphically to $\mathbb{C}$ and $$M_\omega(s)
		=\frac{\Res{z=0}\omega}{s}+\FP\int_0^\infty\omega+O(s).$$
	\end{corollary}
	\begin{proof}
		This is a special case of above lemma with $k_1 = 1$ and $N=1$.
	\end{proof}
	
	\subsection{Values and derivatives at nonpositive integers}
	Let $G\in \mathcal{R}_N$.
	Use the convention
	$$\dfrac{d\y}{\y}:=\prod_{i=1}^N \dfrac{dy_i}{y_i},
	\qquad \y^\alpha:=\prod_{i=1}^N y_i^{\alpha}$$
	for any $\alpha\in \mathbb{C}$.
	Recall the subsitution $\pi$ defined in equation \eqref{eq:pi-def}. For $\sigma\in \mathfrak S_N$ and $n\in\mathbb{Z}_{\ge 0}$,
	set
	$$\omega_{\sigma,n}:=\pi^*\left(G^\sigma(\y)\y^{-n-1}\,d\y\right).$$
	Lemma \ref{lem:G-pullback} shows that $\omega_{\sigma,n}$ is admissible.
	For $1\le j\le N$, define
	$$\rho_{\sigma,j,n}(z_j)\,dz_j
	:=\Res{z_i=0\ (i\ne j)}\omega_{\sigma,n},\qquad
	\mathcal{I}_{\sigma,j,n}
	:=\FP\int_0^{R_j}\rho_{\sigma,j,n}(z_j)\,dz_j,$$
	where $R_1=\infty$ and $R_j=1$ ($j\ge2$). 
	
	Let $$H_n:=\sum_{m=1}^n\dfrac1m,\qquad H_0:=0,$$
	and let $\gamma$ be Euler's constant.
	
	\begin{proposition}\label{prop:I-values}
		The integral
		$$I_G(s):=\dfrac{1}{\Gamma(s)^N}
		\int_{(0,\infty)^N}G(\y)\y^{s}\,\dfrac{d\y}{\y}$$
		is absolutely convergent for $\Re(s)\gg 0$ and admits a meromorphic continuation to $\mathbb{C}$.
		Moreover, for $n\in\mathbb{Z}_{\ge0}$, $I_G$ is holomorphic at $s=-n$ and
		\begin{align}
			I_G(-n)&=\dfrac{(-1)^{Nn}(n!)^N}{N!}
			\sum_{\sigma\in \mathfrak S_N}
			\Res{\z=0}\pi^*\left(
			G^\sigma(\y)\y^{-n-1}\,d\y
			\right),
			\label{eq:I-value}\\
			I_G'(-n)&=\dfrac{(-1)^{Nn}(n!)^N}{N!}
			\left(\sum_{\sigma\in \mathfrak S_N}\sum_{j=1}^N(N-j+1)\mathcal{I}_{\sigma,j,n}\right)
			-N(H_n-\gamma)I_G(-n).
			\label{eq:I-derivative}
		\end{align}
	\end{proposition}
	
	\begin{proof}
		Up to sets of measure zero, decompose $(0,\infty)^N$
		according to the ordering of the coordinates.
		After permuting coordinates, we obtain
		$$I_G(s)
		=\dfrac{1}{\Gamma(s)^N}
		\sum_{\sigma\in\mathfrak{S}_N}
		\int_{0<y_N<\cdots<y_1} G^\sigma(\y)\y^s\dfrac{d\y}{\y}.$$
		Under change of variables $\y=\pi(\z)$,
		\begin{equation}\label{eq:I-sector-decomposition}
			I_G(s)=\dfrac{1}{\Gamma(s)^N} \sum_{\sigma\in\mathfrak{S}_N}
			\int_{X_N} \pi^*\left(G^\sigma(\y)\y^s\dfrac{d\y}{\y}\right).
		\end{equation}
		Lemma \ref{lem:first-two-coefficients} gives
		the convergence and the meromorphic continuation of $I_G$.
		
		Fix $n\in\mathbb{Z}_{\ge0}$. From \eqref{eq:I-sector-decomposition},
		$$I_G(s-n)=\dfrac{1}{\Gamma(s-n)^N}
		\sum_{\sigma\in\mathfrak{S}_N}\int_{X_N} \prod_{j=1}^N z_j^{(N-j+1)s} \omega_{\sigma,n}.$$
		Applying Lemma \ref{lem:first-two-coefficients} with $k_j=N-j+1$ gives
		$$
		\int_{X_N} \prod_{j=1}^N z_j^{(N-j+1)s} \omega_{\sigma,n}
		=\dfrac{\Res{\z=0}\omega_{\sigma,n}}{N!s^N}
		+\dfrac{1}{N!s^{N-1}}\sum_{j=1}^N(N-j+1)\mathcal{I}_{\sigma,j,n}
		+O(s^{-N+2}).$$
		On the other hand,
		$$\Gamma(s-n)^{-N}
		=(-1)^{Nn}(n!)^Ns^N\bigl(1-N(H_n-\gamma)s+O(s^2)\bigr).$$
		Multiplying these expansions shows first that $I_G$ is holomorphic at $s=-n$, and then gives
		the stated formulas for $I_G(-n)$ and $I_G'(-n)$.
	\end{proof}
	
	\subsection{Application: parity-vanishing criterion}
	We conclude this section by giving a sufficient condition
	for the vanishing of a class of Shintani-type zeta functions at negative integers.
	This result will not be used elsewhere in the paper, so this subsection may be skipped.
	
	Let $L_1,\ldots,L_N:\mathbb{R}^r\to\mathbb{R}$
	be nonzero linear forms with non-negative coefficients.
	Write $L_i(\boldsymbol{x})=\sum_{j=1}^{r} a_{i,j}x_j$.
	Assume that for each $1\le j\le r$
	there exists $1\le i\le N$ such that $a_{i,j}\ne 0$. 
	Define the Shintani-type zeta function
	$$
	Z_{L}(s)=\sum_{n_1,\ldots,n_r\in\mathbb{N}}
	\dfrac{1}{(L_1(\boldsymbol{n})\cdots L_N(\boldsymbol{n}))^s},
	\qquad
	\Re(s)\gg 0.$$
	For $1\le j\le r$, put $L_j^*(\y)=\sum_{i=1}^N a_{i,j}y_i$.
	Let
	\[
	G_L(\y)=\prod_{j=1}^r \dfrac{1}{e^{L_j^*(\y)}-1}.
	\]
	Then $G_L\in\mathcal{R}_N$ (see Example \ref{ex:R-functions})
	and $Z_{L}(s)=I_{G_{L}}(s)$.
	Hence Proposition \ref{prop:I-values} implies that
	$Z_{L}(s)$ admits a meromorphic continuation to $\mathbb{C}$
	and is holomorphic at nonpositive integers.
	
	We now restate and prove Theorem \ref{thm:intro-parity-vanishing}.
	\begin{theorem} \label{thm:parity-vanishing}
		Assume that $N\ge r$ and $L_i(\boldsymbol{x})=x_i$ for each $1\le i\le r$.
		Let $n\in\mathbb{N}$.
		If $Nn+r\equiv 1 \pmod{2}$, then
		$Z_{L}(-n)=0$.
	\end{theorem}
	
	\begin{proof}
		By \eqref{eq:I-value}, it suffices to show that, for every $\sigma\in\mathfrak{S}_N$,
		\begin{equation} \label{eq:res-to-vanish}
			\Res{\z=0}\pi^*\left(
			G_{L}^\sigma(\y)\y^{-n-1}\,d\y
			\right)=0.
		\end{equation}
		Expand each factor in $G_L(\y)$ by
		\[
		\frac{1}{e^u-1}
		=\sum_{k=0}^{\infty}\frac{B_k}{k!}u^{k-1}.
		\]
		Consequently, we have 
		$G_L(\y)=\sum_{\boldsymbol{k}\in \mathbb{Z}_{\ge 0}^r} S_{\boldsymbol{k}}(\y)$,
		where
		\[
		S_{\boldsymbol{k}}(\y)
		=\prod_{j=1}^r \frac{B_{k_j}}{k_j!}
		\left(L_j^*(\y)\right)^{k_j-1}.
		\]
		Write each residue as a contour integral over sufficiently small circle around the origin.
		Then the sum over $\boldsymbol{k}$ converges absolutely and uniformly on the resulting torus.
		Hence the residue may be evaluated termwise.
		
		Since the exponent of $z_1$ in $\pi^*(S_{\boldsymbol{k}}^\sigma(\y)\y^{-n})$ is
		$\sum_{j=1}^r (k_j-1)-Nn$ and $\pi^*(d\y/\y)=d\z/\z$,
		a necessary condition for a nonzero contribution is
		\begin{equation}
			\label{eq:degree-condition}
			\sum_{j=1}^r k_j=Nn+r.
		\end{equation}
		Thus, it remains to consider only those $\boldsymbol{k}$
		satisfying \eqref{eq:degree-condition} for which $\prod_{j=1}^r B_{k_j}\ne 0$.
		Since $B_k=0$ for every odd integer $k\ge3$,
		each $k_j$ is either $1$ or an even non-negative integer.
		Since $Nn+r$ is odd,
		\eqref{eq:degree-condition} implies that $k_{j_0}=1$ for some $j_0\in \{1,\ldots,r\}$.
		
		By the assumption $L_i(\boldsymbol{x})=x_i$ for $1\le i\le r$,
		the variable $y_{j_0}$ occurs only in $L_{j_0}^*(\y)$.
		Since $k_{j_0}-1=0$, $S_{\boldsymbol{k}}(\y)$ is independent of $y_{j_0}$,
		and hence $S_{\boldsymbol{k}}^\sigma (\y)$ is independent of $y_{i_0}$, where $i_0:=\sigma(j_0)$.
		Since $n\ge 1$, we obtain
		\[
		\Res{\z=0}\pi^*\left(S_{\boldsymbol{k}}^\sigma(\y)\y^{-n-1}\,d\y\right)
		=\Res{y_1=0}\cdots \Res{y_{i_0}=0}\cdots\Res{y_N=0}
		S_{\boldsymbol{k}}^{\sigma} (\y)\y^{-n-1}\,d\y
		=0.
		\]
		Therefore, \eqref{eq:res-to-vanish} holds for every $\sigma\in\mathfrak{S}_N$.
		This proves Theorem \ref{thm:parity-vanishing}.
	\end{proof}
	
	The theorem shows that every odd-rank Witten zeta function $\xi_\Phi(s)$ vanishes at every negative even integer. Also, for the root system $\Phi = A_r$ ($r\equiv 2,3 \pmod{4}$) or $D_r$ ($r$ odd), it vanishes at every negative odd integers. Both facts are special cases of a general vanishing result proved in \cite{au2024vanishing}.
    Recently, Rutard and Shinohara \cite{Rutard2026trivial} independently obtained analogous vanishing results for certain $A_r$-multivariable zeta functions.
	
	\section{Conical zeta function}
	Throughout this section, $V=\mathbb R^r$ is equipped with the lattice $\mathbb Z^r$.
	Let $\L:=(L_1,\ldots,L_N)$ be a multiset of nonzero rational linear forms on $V$, with repetitions allowed, and let $\A:=\{H_i=\ker L_i:1\leq i\leq N\}$ be its underlying hyperplane arrangement. We assume that $\A$ is essential, or equivalently that the forms $L_i$ span $V^\ast$. Define the (complete) \textit{conical zeta function}
	$$\xi_\L(s) := \sum_{\boldsymbol{x} \in \mathbb{Z}^r}^{} {'}\frac{1}{|L_1(\boldsymbol{x})L_2(\boldsymbol{x})\cdots L_N(\boldsymbol{x})|^s},\qquad \Re(s)\gg 0,$$
	where $\sum'$ means that we omit those $\boldsymbol{x}$ for which at least one of the linear forms vanishes.\par
	
	For each chamber $D$ of $\mathcal{A}$, define the (partial) \textit{conical zeta function}
	$$Z_{\L,D}(s) := \sum_{\boldsymbol{x} \in \mathbb{Z}^r \cap D}\frac{1}{|L_1(\boldsymbol{x}) L_2(\boldsymbol{x}) \cdots L_N(\boldsymbol{x})|^s},\qquad \Re(s)\gg 0.$$
	Then evidently
	$$\xi_\L(s) = \sum_{D\in \Ch(\A)} Z_{\L,D}(s).$$

	For each chamber $D$ of $\A$, let $$G_{\L,D}(y_1,\ldots,y_N) = \sum_{\boldsymbol{x}\in \mathbb{Z}^r\cap D} \exp\left(-\sum_{i=1}^N y_i |L_i(\boldsymbol{x})|\right).$$
	
	For given $D\in \Ch(\A)$, there exists $(\varepsilon_1(D),\ldots,\varepsilon_N(D)) \in \{\pm 1\}^N$ such that $$D = \{\boldsymbol{x}\in \mathbb{R}^r: \varepsilon_i(D) L_i(\boldsymbol{x}) > 0\}.$$
	We let $l_D=l_D(\y)$ to be linear functional $\sum_{i=1}^N y_i \varepsilon_i(D) L_i$. Since $l_D$ is positive on $\overline{D}\setminus\{0\}$, the extension of Proposition \ref{prop:Hilbert-series-properties}(b) to half-open cones applies. Hence
	$$G_{\L,D}(\y) = \sum_{\boldsymbol{x}\in \mathbb{Z}^r \cap D} \exp\left(-\sum_{i=1}^N y_i \varepsilon_i(D)L_i(\boldsymbol{x})\right) = \sum_{\boldsymbol{x}\in \mathbb{Z}^r \cap D} e^{-l_D(\boldsymbol{x})} = \mathrm H_D(l_D).$$
	Here the last member is the evaluation, in the sense of Section~2.2, of the group-algebra element $\mathrm H_D$ at the dual vector $l_D=l_D(\y)$.
	This shows $G_{\L,D}(\y)$ is an exponential rational function with exponential decay at infinity, and hence $G_{\L,D}\in \mathcal{R}_N$ as in Section 3. We also have
	$$Z_{\L,D}(s) = I_{G_{\L,D}}(s) = \frac{1}{\Gamma(s)^N} \int_{(0,\infty)^N} \mathrm H_D(l_D) \y^{s-1} d\y.$$
	Let $$G_\L(\y) := \sum_{D\in \Ch(\A)} G_{\L,D}(\y) = \sum_{D\in \Ch(\A)} \mathrm H_D(l_D),$$
	then we obtain the following integral representation of the conical zeta function
	\begin{equation}\label{conical_zeta_int_rep}\xi_\L(s) = \frac{1}{\Gamma(s)^N} \int_{(0,\infty)^N} G_\L(\y) \y^{s-1} d\y.\end{equation}
	
	\begin{example}\label{Ex_1}
		Let $\Lambda=\mathbb{Z}^2=\mathbb{Z}v_1\oplus\mathbb{Z}v_2$, where $v_1=(1,0)$ and $v_2=(0,1)$ are the standard unit vectors. Thus $\mathbb{Z}[\Lambda] = \mathbb{Z}[\e^{\pm v_1},\e^{\pm v_2}]$. Consider the collection of integral linear forms $\mathcal{L} = \{L_1,L_2,L_3\}$ with
		$$L_1(x_1v_1 + x_2v_2)=x_1,\qquad L_2(x_1v_1 + x_2 v_2)=x_2,\qquad L_3(x_1v_1 + x_2v_2)=x_1+2x_2.$$
		For $\varepsilon_i\in\{+,-\}$, let $D_{\varepsilon_1\varepsilon_2\varepsilon_3}$
		denote the chamber on which the signs of
		$(L_1,L_2,L_3)$ are $(\varepsilon_1,\varepsilon_2,\varepsilon_3)$.
		The arrangement has six chambers. Their Hilbert series and the associated functions $G_{\L,D}(\y)$ are shown in the table below.
		\begin{center}
			\begingroup
			\footnotesize
			\setlength{\tabcolsep}{3pt}
			\renewcommand{\arraystretch}{2}
			\begin{tabular}{c|c|c}
				$D$ & $\mathrm H_D$ & $G_{\L,D}(\y)$ \\
				\hline
				$D_{+++}$ &
				$\dfrac{\e^{v_1+v_2}}{(1-\e^{v_1})(1-\e^{v_2})}$ &
				\multirow{2}{*}{$\dfrac{e^{-(y_1+y_3)}}{1-e^{-(y_1+y_3)}}
					\dfrac{e^{-(y_2+2y_3)}}{1-e^{-(y_2+2y_3)}}$} \\
				\cline{1-2}
				$D_{---}$ &
				$\dfrac{\e^{-v_1-v_2}}{(1-\e^{-v_1})(1-\e^{-v_2})}$ & \\
				\hline
				$D_{-++}$ &
				$\dfrac{\e^{-v_1+v_2}(1+\e^{-v_1+v_2})}{(1-\e^{v_2})(1-\e^{-2v_1+v_2})}$ &
				\multirow{2}{*}{$\dfrac{e^{-(y_1+y_2+y_3)}\bigl(1+e^{-(y_1+y_2+y_3)}\bigr)}
					{\bigl(1-e^{-(2y_1+y_2)}\bigr)\bigl(1-e^{-(y_2+2y_3)}\bigr)}$} \\
				\cline{1-2}
				$D_{+--}$ &
				$\dfrac{\e^{v_1-v_2}(1+\e^{v_1-v_2})}{(1-\e^{-v_2})(1-\e^{2v_1-v_2})}$ & \\
				\hline
				$D_{-+-}$ &
				$\dfrac{\e^{-3v_1+v_2}}{(1-\e^{-v_1})(1-\e^{-2v_1+v_2})}$ &
				\multirow{2}{*}{$\dfrac{e^{-(y_1+y_3)}}{1-e^{-(y_1+y_3)}}
					\dfrac{e^{-(2y_1+y_2)}}{1-e^{-(2y_1+y_2)}}$} \\
				\cline{1-2}
				$D_{+-+}$ &
				$\dfrac{\e^{3v_1-v_2}}{(1-\e^{v_1})(1-\e^{2v_1-v_2})}$ & \\
			\end{tabular}
			\endgroup
		\end{center}
		We illustrate the computation of $\mathrm{H}_D$ for $D=D_{+--}$. In the coordinates $av_1+bv_2$, this chamber is given by $a>0$, $b<0$, and $a+2b<0$, and hence
		$$D=\mathbb{R}_{>0}(-v_2)+\mathbb{R}_{>0}(2v_1-v_2).$$
		We compute its Hilbert series directly from the definition. Put $u=a$ and $v=-(a+2b)$. Then $av_1+bv_2=u v_1-(u+v)v_2/2$, and $\mathbb{Z}^2\cap D$ corresponds to $u,v\in\mathbb{N}$ satisfying $u\equiv v\pmod 2$. Hence
		$$\mathrm H_D=\sum_{\substack{u,v\geq1\\u\equiv v\pmod 2}}\e^{u v_1-(u+v)v_2/2}.$$
		For the odd-odd terms, write $u=2m+1$ and $v=2n+1$ with $m,n\geq0$; for the even-even terms, write $u=2m$ and $v=2n$ with $m,n\geq1$. Summing the resulting geometric series gives
		$$
		\begin{aligned}
			\mathrm H_D
			&=\sum_{m,n\geq0}\e^{(2m+1)v_1-(m+n+1)v_2}+\sum_{m,n\geq1}\e^{2mv_1-(m+n)v_2}\\
			&=\frac{\e^{v_1-v_2}}{(1-\e^{-v_2})(1-\e^{2v_1-v_2})}+\frac{\e^{2v_1-2v_2}}{(1-\e^{-v_2})(1-\e^{2v_1-v_2})}\\
			&=\frac{\e^{v_1-v_2}(1+\e^{v_1-v_2})}{(1-\e^{-v_2})(1-\e^{2v_1-v_2})}.
		\end{aligned}
		$$
		Since the signs of $(L_1,L_2,L_3)$ on $D$ are $(+,-,-)$, we have $l_D=y_1L_1-y_2L_2-y_3L_3$. Consequently,
		$$\e^{v_1-v_2}\longmapsto e^{-(y_1+y_2+y_3)},\qquad \e^{-v_2}\longmapsto e^{-(y_2+2y_3)},\qquad \e^{2v_1-v_2}\longmapsto e^{-(2y_1+y_2)}.$$
		Evaluating the preceding expression for $\mathrm H_D$ at $l_D$ gives the displayed formula for $G_{\L,D_{+--}}(\y)$.
		
		In this example, Lemma~\ref{hilbert_chamber_identity} is the following concrete identity in $\mathbb{Q}(\Lambda)$:
		$$
		\begin{aligned}
			\mu_\A(\mathbb{R}^2,\{0\})= 2={}&\frac{\e^{v_1+v_2}}{(1-\e^{v_1})(1-\e^{v_2})}+\frac{\e^{-v_1-v_2}}{(1-\e^{-v_1})(1-\e^{-v_2})}\\
			&+\frac{\e^{-v_1+v_2}(1+\e^{-v_1+v_2})}{(1-\e^{v_2})(1-\e^{-2v_1+v_2})}+\frac{\e^{v_1-v_2}(1+\e^{v_1-v_2})}{(1-\e^{-v_2})(1-\e^{2v_1-v_2})}\\
			&+\frac{\e^{-3v_1+v_2}}{(1-\e^{-v_1})(1-\e^{-2v_1+v_2})}+\frac{\e^{3v_1-v_2}}{(1-\e^{v_1})(1-\e^{2v_1-v_2})}.
		\end{aligned}
		$$
		Therefore we see the complete conical zeta function $$\xi_\L(s) =\sum_{(x_1,x_2) \in \mathbb{Z}^2}^{} {'} \frac{1}{|x_1 x_2 (x_1+2x_2)|^s}$$ has the integral representation $$\xi_\L(s) = \frac{1}{\Gamma(s)^3} \int_{(0,\infty)^3} G_\L(\y) \y^{s-1} d\y$$
		with $$G_\L(\y) = \frac{2e^{-(y_1+y_3)}}{1-e^{-(y_1+y_3)}}
		\frac{e^{-(y_2+2y_3)}}{1-e^{-(y_2+2y_3)}} + \frac{2e^{-(y_1+y_2+y_3)}\bigl(1+e^{-(y_1+y_2+y_3)}\bigr)}
		{\bigl(1-e^{-(2y_1+y_2)}\bigr)\bigl(1-e^{-(y_2+2y_3)}\bigr)} + \frac{2e^{-(y_1+y_3)}}{1-e^{-(y_1+y_3)}}
		\frac{e^{-(2y_1+y_2)}}{1-e^{-(2y_1+y_2)}}.$$
	\end{example}
	This example shows that writing out the explicit expressions for $G_\L(\y)$ in full is generally complicated. Our analysis of $\xi_\L(s)$ does not depend on such an explicit expression.
	
	We now prove Theorem \ref{thm:conical_0}, in the slightly stronger form below.
	\begin{theorem}\label{values_at_nonpositive_even}
		We have
		$$\xi_\L(0) = \mu_\A(V,\{0\}),\qquad \xi_\L(-2n) = 0 \text{ for }n\in \mathbb{N}.$$
	\end{theorem}
	\begin{proof}
		For $\sigma\in\mathfrak{S}_N$, put (see \eqref{eq:sym-group-action} for the action of symmetric group)
		$$l_D^\sigma:=\sum_{i=1}^N y_{\sigma(i)}\varepsilon_i(D)L_i,
		\qquad l^\sigma:=\sum_{i=1}^N y_{\sigma(i)}L_i.$$
		Applying Proposition \ref{prop:I-values} to \eqref{conical_zeta_int_rep}, we have
		$$\begin{aligned}\xi_\L(-n) &=\frac{(-1)^{Nn}(n!)^N}{N!}\sum_{\sigma\in\mathfrak S_N}
			\Res{\z=0}\pi^*\!\left(G^\sigma_\L(\y)\y^{-n-1}d\y\right) \\
			&= \frac{(-1)^{Nn}(n!)^N}{N!}\sum_{\sigma\in\mathfrak S_N} \Res{\z=0}\pi^*\!\left(
			\left(\sum_{D\in\Ch(\A)}\mathrm H_D(l_D^\sigma)\right)
			\y^{-n-1}d\y\right).\end{aligned}$$
		When $n$ is even, we can replace $l_D^\sigma$ into $l^\sigma$ inside the summation, because the residue is unchanged when any $y_i$ is replaced by $-y_i$. Hence
		\begin{equation}\label{aux_15}
			\xi_\L(-n)=\frac{(-1)^{Nn}(n!)^N}{N!}\sum_{\sigma\in\mathfrak S_N}\Res{\z=0}\pi^*\!\left(
			\left(\sum_{D\in\Ch(\A)}\mathrm H_D(l^\sigma)\right)\y^{-n-1}d\y\right).\end{equation}
		Lemma~\ref{hilbert_chamber_identity} gives the formal identity
		$$\sum_{D\in\Ch(\A)}\mathrm H_D=\mu_\A(V,\{0\})$$
		in $\mathbb Q(\Lambda)$. Evaluating it at the dual vector $l^\sigma$ gives, as an identity of exponential rational functions,
		$$\sum_{D\in\Ch(\A)}\mathrm H_D(l^\sigma)=\mu_\A(V,\{0\}).$$
		Substitute it back into \eqref{aux_15} and note that
		$$\pi^*\!\left(\y^{-n-1}d\y\right)=\prod_{j=1}^Nz_j^{-n(N-j+1)-1}\,d\z.$$
		Its residue is $1$ for $n=0$ and $0$ for every positive $n$, completing the proof.
	\end{proof}
	
	A curious and surprising consequence of this theorem is the following "multiplicity-independent" result, which is in general false for partial conical zeta $Z_{\L,D}(s)$ and Shintani zeta function; see, for example, \cite[Theorem 2.7]{borwein2018derivatives}.
	\begin{corollary}
		Let $L_1,\ldots,L_N$ be nonzero rational linear forms spanning the dual space, and $n_1,\ldots,n_N$ be positive integers. Then the value of following function at $s=0$ is independent of $n_1,\ldots,n_N$:
		$$\sum_{\boldsymbol{x} \in \mathbb{Z}^r}^{} {'}\frac{1}{|L_1(\boldsymbol{x})^{n_1} L_2(\boldsymbol{x})^{n_2}\cdots L_N(\boldsymbol{x})^{n_N}|^s}.$$
	\end{corollary}
	\begin{proof}
		Let $\mathcal{L}$ be the multi-set formed from $L_1,\ldots, L_N$ with each $L_i$ occurring $n_i$ times. Then the quantity concerned is $\xi_\L(0) = \mu_\A(V,\{0\})$. The RHS depends only on the hyperplane arrangement $\A = \{\ker L_1,\ldots,\ker L_N\}$, which forgets the multiplicity $n_i$. 
	\end{proof}
	
	\subsection{Witten zeta function}
	Given a root system $\Phi$, we adopt the notations as in Section \ref{root_system_subsection}. Consider the following $N$ integral linear forms from $\mathbb{R}^r$ to $\mathbb{R}$:
	$$L_\gamma(x_1,\ldots,x_r) := \sum_{j=1}^r x_j (\lambda_j, \gamma^\vee),\qquad \gamma\in \Phi^+.$$
	Then $\ker L_\gamma = H_\gamma$ are the reflecting hyperplane of the roots. The partial conical zeta functions are the same for all chambers. Define
	$$\xi_\Phi(s) :=\frac{1}{|W|}\sum_{\boldsymbol{x}\in \mathbb{Z}^r}^{} {'} \frac{1}{\prod_{\gamma\in \Phi^+} |L_\gamma(\boldsymbol{x})|^s} =\sum_{\boldsymbol{x}\in \mathbb{N}^r}^{} \frac{1}{\prod_{\gamma\in \Phi^+} L_\gamma(\boldsymbol{x})^s}.$$
	Thus $\xi_{\Phi}(s)$ is, up to multiplication by $|W|$, the conical zeta function $\xi_\L(s)$ associated with the family $\L=(L_\gamma)_{\gamma\in\Phi^+}$; its underlying arrangement $\A$ is the Weyl arrangement.
	\begin{remark}
		Let $G$ be a simply-connected compact Lie group, $\Phi$ be its root system. The function $\xi_\Phi(s)$ is closely related to the representative-theoretic object $$\zeta_G(s) = \sum_{\rho} \frac{1}{(\text{dim } \rho)^s},\quad \rho\in \{\text{finite-dimensional irreducible complex representations of $G$}\}.$$
		More precisely, we have
		$$\xi_\Phi(s) = K_\Phi^{-s} \zeta_G(s),$$
		where the normalization constant $K_\Phi \in \mathbb{N}$
		is given by 
		\begin{equation}\label{aux_5}K_\Phi = \prod_{\gamma\in \Phi^+} (\rho, \gamma^\vee) = \prod_{i=1}^r e_i! \qquad \rho:= \frac{1}{2}\sum_{\gamma\in \Phi^+} \gamma \text{ is the Weyl vector}.\end{equation}
		Thus $\xi_\Phi(s)$ defined above agrees with the normalized Witten zeta function introduced in Section 1.
	\end{remark}
	
	\begin{proof}[Proof of Theorem \ref{thm:witten_0}]
		The result follows by combining Theorem \ref{values_at_nonpositive_even} and Lemma \ref{prod_exp}.
	\end{proof}
	
	\begin{remark}
		(a) Theorem \ref{values_at_nonpositive_even} also implies that
		$$\xi_\Phi(-2n)=0,$$
		i.e., the Witten zeta function vanishes at every negative even integer. The case $n=1$ gives, for simply-connected compact Lie groups, another proof of the vanishing at $s=-2$ conjectured by Kurokawa and Ochiai \cite{kurokawa2013zeros}. When the rank $r$ is odd, the vanishing at every negative even integer also follows directly from the parity criterion in Theorem \ref{thm:parity-vanishing}, and hence is a consequence of a more general parity phenomenon for Shintani-type zeta functions.
		
		More strongly, for Witten zeta functions, it is proved in \cite{au2024vanishing} that the order of vanishing at each such point is at least $r$. The same argument and conclusion extend to $\xi_\L(s)$ for general $\L$.
		\par (b) When \(n\) is positive and odd, the proof above does not provide further insight about the value $\xi_\L(-n)\in \mathbb{Q}$. It is in general non-zero. Proposition \ref{prop:I-values} gives a consice but inefficient formula for computing it: the formula involves a sum of \(N!\) residues, each of which is itself difficult to evaluate. Even for Witten zeta functions, the efficient computation of \(\xi_\Phi(-n)\) remains largely open and is currently understood only for root systems of rank two \cite{au2024single}. These values may be of particular interest because they appear to possess certain $p$-adic properties analogous to $\zeta(-n)$ of the Riemann zeta function.
	\end{remark}

	\subsection{Chamber sums and partial residues}
	We shall need the following two propositions in the next section. For a $\Psi\in \mathcal{R}_N$ and the map $\pi$ as in \eqref{eq:pi-def}, define
	$$\mathscr R_j(\Psi):=\Res{z_k=0\ (k\ne j)}\pi^*\left(\Psi(\y)\frac{d\y}{\y}\right),$$
	where the residue is taken with respect to variables $z_1,\ldots,z_{j-1},z_{j+1},\ldots,z_N$, so that $\mathscr R_j(\Psi)$ is a one-form with respect to the remaining variable $z_j$. \par
	Since $\pi^*(d\y/\y)=d\z/\z$, changing any of the $N-1$ $\z$-variables other than $z_j$ to its negative does not affect $\mathscr R_j(\Psi)$. A change of signs $y_m\mapsto\delta_my_m$ induces
	\begin{equation}\label{aux_10}z_1\mapsto\delta_1z_1,\qquad z_m\mapsto\delta_m\delta_{m-1}z_m\quad(m\geq2).\end{equation}
	Therefore we see
	\begin{itemize}[leftmargin=*]
		\item When $j=1$, changing any of the $y_2,\ldots,y_N$ to its negative does not affect $\mathscr R_j(\Psi)$.
		\item When $j\geq 2$, changing any of the $y_1,\ldots,y_N$ to its negative, with the condition that $y_{j-1},y_j$ change by same sign, does not affect $\mathscr R_j(\Psi)$.
	\end{itemize}
	
	Let $$H_i^+=\{x\in V:L_i(x)>0\},\qquad H_i^-=\{x\in V:L_i(x)<0\}.$$
	For a one-dimensional flat $X\not\subset H_i$, let $u_{X,i}$ be the lattice vector determined by
	$$X\cap H_i^+\cap\mathbb Z^r= \mathbb{N} u_{X,i}.$$
	\begin{proposition}\label{one_sided_residue}
		For $\sigma \in \mathfrak{S}_N$, let $k = \sigma^{-1}(1)$, then
		$$\mathscr{R}_1(G_\L^\sigma)=\sum_{\substack{X\in L(\A)\\\dim X=1, \ X\not\subset H_k}}
		\frac{2\mu_\A(V,X)}{\exp(z_1L_k(u_{X,k}))-1}\frac{dz_1}{z_1}.$$
	\end{proposition}
	\begin{proof}
		For $D\in\Ch(\A)$, define
		$$p_+^\sigma=y_1L_k+\sum_{i\ne k}y_{\sigma(i)}L_i,\qquad 
		p_-^\sigma=-y_1L_k+\sum_{i\ne k}y_{\sigma(i)}L_i.$$
		As we remarked, changing signs of $y_2,\ldots,y_N$ does not affect $\mathscr{R}_1$, it follows that
		\begin{align}\label{aux_1}
			\mathscr{R}_1(G_\L^\sigma) &= \mathscr{R}_1\left(\sum_{D\in\Ch(\A)}\mathrm H_D(l_D^\sigma)\right) \nonumber \\ 
			&= \mathscr{R}_1
			\left(\sum_{D\subset H_k^+}\mathrm H_D(p_+^\sigma)+\sum_{D\subset H_k^-}\mathrm H_D(p_-^\sigma)\right).
		\end{align}
		On the other hand, multiplying both sides of the identity in Proposition \ref{indicator_identity} by the indicator function $\ind_{H_k^\pm}$ gives
		$$\sum_{D\subset H_k^\pm}  \ind_{D} =\sum_{\substack{X\in L(\A)\\X\not\subset H_k}}
		\mu_\A(V,X)\ind_{X\cap H_k^\pm}.$$
		Applying the Hilbert-series valuation gives the formal identity
		$$\sum_{D\subset H_k^\pm}\mathrm H_D
		=\sum_{\substack{X\in L(\A)\\X\not\subset H_k}}
		\mu_\A(V,X)\mathrm H_{X\cap H_k^\pm}$$
		in $\mathbb Q(\Lambda)$. If $\dim X\geq2$, then $\dim(X\cap H_k)\ge 1$, and $X\cap H_k^\pm$ is invariant under translation by every vector in $X\cap H_k$. Lemma~\ref{pointed_vanishing_lemma} therefore gives $\mathrm H_{X\cap H_k^\pm}=0$. Thus only the case $\dim X=1$ contributes. For such an $X$, the definition of $u_{X,k}$ gives
		$$X\cap H_k^\pm \cap\mathbb Z^r= \pm \mathbb{N}\; u_{X,k}.$$
		Hence $$\mathrm H_{X\cap H_k^+} = \sum_{n\geq 1} \e^{n u_{X,k}} = \frac{\e^{u_{X,k}}}{1-\e^{u_{X,k}}} = \frac{1}{\e^{-u_{X,k}}-1},\qquad \mathrm H_{X\cap H_k^-} = \frac{1}{\e^{u_{X,k}}-1}.$$
		Evaluating them at $p_\pm^\sigma$ yields
		$$\mathrm H_{X\cap H_k^+}(p_+^\sigma)=\frac{1}{\exp(p_+^\sigma(u_{X,k}))-1},\qquad \mathrm H_{X\cap H_k^-}(p_-^\sigma)=\frac{1}{\exp(p_-^\sigma(-u_{X,k}))-1}.$$
		Substitute this to \eqref{aux_1} gives
		$$\mathscr{R}_1(G_\L^\sigma) = \sum_{\substack{X\in L(\A)\\ \dim X = 1, \ X\not\subset H_k}} \mu_\A(V,X) \mathscr{R}_1 \left(\frac{1}{\exp(p_+^\sigma(u_{X,k}))-1}  + \frac{1}{\exp(p_-^\sigma(-u_{X,k}))-1}
		\right).$$
		It is easy to evaluate the two residues using the observation that
		$\pi^*\!\left(\frac{d\y}{\y}\right)
		=\frac{d\z}{\z}$ and
		$$\Res{z_2=\cdots=z_N=0}\frac{1}{\exp\left(az_1+\sum_{m=2}^Nc_mz_1\cdots z_m\right)-1}
		\frac{d\z}{\z} =\frac{1}{\exp(az_1)-1}\frac{dz_1}{z_1},\qquad a\in\mathbb{Q}^\times, \ c_2,\ldots,c_N\in \mathbb{Q}.$$
		Applying this with $a=L_k(u_{X,k})$ says that both residues equal $(\exp(z_1L_k(u_{X,k}))-1)^{-1} \frac{dz_1}{z_1}$, completing the proof.
	\end{proof}
	
	\begin{proposition}\label{laurent_partial_residue}
		Fix $2\leq j\leq N$, we have
		$$\mathscr{R}_j(G_\L^\sigma) \in (\mathbb{Q}z_j^{-1}+\mathbb{Q}+\mathbb{Q}z_j)\frac{dz_j}{z_j}.$$\end{proposition}
	\begin{proof}		
		\noindent\emph{Step 1: fix the chamber signs.}
		Put $a:=\sigma^{-1}(j-1), b:=\sigma^{-1}(j)$ and, for $D\in\Ch(\A)$, set
		$$s(D):=\varepsilon_a(D)\varepsilon_b(D)\in\{\pm1\}.$$
		For a fixed chamber $D$, set $$p_{s(D)}^\sigma
		:=\sum_{m\ne j}y_mL_{\sigma^{-1}(m)}
		+s(D)y_jL_b.$$
		The sign-change observation in \eqref{aux_10} says
		\begin{equation}\label{partial_residue_sign_split}
			\mathscr R_j(G_\L^\sigma)
			=\sum_{s\in \{1,-1\}}\mathscr R_j\left(
			\sum_{\substack{D\in\Ch(\A)\\s(D)=s}}
			\mathrm H_D(p_s^\sigma)\right).
		\end{equation}
		
		\smallskip
		\noindent\emph{Step 2: apply Proposition \ref{indicator_identity}.}
		Let
		$$U_s:=\{x\in V:sL_a(x)L_b(x)>0\}.$$
		The condition $s(D)=s$ is equivalent to $D\subset U_s$. Multiplying
		Proposition~\ref{indicator_identity} by $\ind_{U_s}$ and applying the Hilbert-series valuation yields the formal identity
		$$\sum_{\substack{D\in\Ch(\A)\\s(D)=s}}
		\mathrm H_D
		=\sum_{X\in L(\A)}\mu_\A(V,X)
		\mathrm H_{X\cap U_s}$$
		in $\mathbb Q(\Lambda)$. Evaluating this identity at the dual vector $p_s^\sigma$ gives
		$$\sum_{\substack{D\in\Ch(\A)\\s(D)=s}}
		\mathrm H_D(p_s^\sigma)
		=\sum_{X\in L(\A)}\mu_\A(V,X)
		\mathrm H_{X\cap U_s}(p_s^\sigma).$$
		Substitution in \eqref{partial_residue_sign_split} gives
		\begin{equation}\label{partial_residue_flat_expansion}
			\mathscr R_j(G_\L^\sigma)
			=\sum_{s\in \{1,-1\}}\sum_{X\in L(\A)}\mu_\A(V,X)
			\mathscr R_j\left(\mathrm H_{X\cap U_s}(p_s^\sigma)\right).
		\end{equation}
		It remains to show every $$\mathscr R_j(\mathrm H_{X\cap U_s}(p_s^\sigma)) = \operatorname{CT}_{z_k=0\ (k\neq j)}(\mathrm H_{X\cap U_s}(p_s^\sigma)) \frac{dz_j}{z_j} \in(\mathbb{Q}z_j^{-1}+\mathbb{Q}+\mathbb{Q}z_j)\frac{dz_j}{z_j},$$
		here the middle expression means constant terms of the Laurent expansion with respect to variables $z_1,\ldots,z_N$ except $z_j$.
		
		\smallskip
		\noindent\emph{Step 3: reduce to flats of dimension at most two.}
		Suppose first that $\dim X\geq3$. Then $\dim(X\cap H_a\cap H_b) \geq 1$, and $X\cap U_s$ is translation invariant under any nonzero vector in this space, so $\mathrm H_{X\cap U_s}=0$ by Lemma~\ref{pointed_vanishing_lemma}. Suppose that $\dim X=2$ but $L_a|_X$ and $L_b|_X$ are linearly dependent. If both restrictions vanish, then $X\cap U_s$ is empty; otherwise it is invariant under translation by a nonzero rational vector in their common kernel. In either case its Hilbert series vanishes. In particular, this applies when $L_a$ and $L_b$ are proportional. Hence only
		one-dimensional flats and two-dimensional flats on which
		$L_a,L_b$ are independent can contribute to
		\eqref{partial_residue_flat_expansion}.
		
		\smallskip
		\noindent\emph{Step 4: the one-dimensional contributions.}
		If $\dim X=1$, then $X\cap U_s$ is either empty or $X-\{0\}$.
		So $\mathrm{H}_{X\cap U_s}$ is either $0$ or $\mathrm{H}_X-\mathrm{H}_{\{0\}}=-1$.
		Thus $\mathscr R_j(\mathrm H_{X\cap U_s}(p_s^\sigma)) \in \mathbb{Q} \frac{dz_j}{z_j}$ in this case.
		
		\smallskip
		\noindent\emph{Step 5: the two-dimensional contributions.}
		Let $\dim X=2$ and suppose that $L_a|_X,L_b|_X$ are independent.
		Write
		\begin{equation}\label{eq:crucial_obs}L_{\sigma^{-1}(m)}|_X=\alpha_mL_a|_X+\beta_mL_b|_X,\qquad
			\alpha_m,\beta_m\in\mathbb{Q}.\end{equation}
		Put
		$$q_1:=1,\qquad q_m:=z_2\cdots z_m\quad(m\geq2).$$
		Since $y_m=z_1q_m$, the restriction of the pullback functional is
		$$\pi^*(p_s^\sigma)|_X
		=z_1\bigl(A_X(\z)L_a+B_{X,s}(\z)L_b\bigr),$$
		where
		$$A_X:=\sum_{m\ne j}\alpha_mq_m,
		\qquad
		B_{X,s}:=\sum_{m\ne j}\beta_mq_m+s q_j.$$
		Since $a=\sigma^{-1}(j-1)$ and $b=\sigma^{-1}(j)$, equation \eqref{eq:crucial_obs}, together with the definition of $B_{X,s}$, gives		\begin{equation}\label{rank_two_distinguished_coefficients}
			[q_{j-1}]A_X=1,\quad[q_j]A_X=0,
			\qquad[q_{j-1}]B_{X,s}=0,\quad[q_j]B_{X,s}=s,
		\end{equation}
		where $[q_m]f$ denotes the coefficient of $q_m$ in the polynomial $f$.
		
		Since $X\cap U_s$ is the disjoint union of the interiors of two simplicial cones, each of which has one extreme ray in $X\cap H_a$ and the other in $X \cap H_b$, we may write (Proposition \ref{prop:Hilbert-series-properties}(c))
		$$\pi^*(\mathrm{H}_{X\cap U_s}(p_s^\sigma))=\frac{g(\z)}{(e^{a_X z_1 A_X}-1)(e^{b_X z_1 B_{X,s}}-1)},$$
		where $a_X,b_X\in\mathbb{Q}^\times$ and $g(\z)$ is a finite sum of $e^{z_1(c A_X+d B_{X,s})}$ ($c,d\in\mathbb{Q}$).
		Thus
		\begin{equation}\label{rank_two_z1_constant_term}
			\operatorname{CT}_{z_1=0}
			\bigl(\pi^*(\mathrm{H}_{X\cap U_s}(p_s^\sigma))\bigr)
			\in{\mathbb Q}+{\mathbb Q}\,\frac{A_X}{B_{X,s}}+{\mathbb Q}\,\frac{B_{X,s}}{A_X}.
		\end{equation}
		
		\smallskip
		\noindent\emph{Step 6: take the remaining constant terms.}
		Call $f\in \mathbb{Q}[z_2,\ldots,z_N]$ nested if it has the form
		$$f=\sum_{m=1}^Nc_mq_m,
		\qquad q_1=1,\quad q_m=z_2\cdots z_m.$$
		Note the following elementary observation: if $f,g$ are nested polynomials and $f$ satisfies either of the two coefficient patterns in
		\eqref{rank_two_distinguished_coefficients}, then
		$$\operatorname{CT}_{z_k=0\, (k\notin\{1,j\})}\frac{g}{f}
		\in\mathbb{Q}z_j^{-1}+\mathbb{Q}+\mathbb{Q}z_j.$$
		
		Both $A_X$ and $B_{X,s}$ are nested polynomials satisfying the two
		respective alternatives in \eqref{rank_two_distinguished_coefficients}.
		Applying the observation to $A_X/B_{X,s}$ and
		$B_{X,s}/A_X$ shows that the remaining constant terms in
		\eqref{rank_two_z1_constant_term} belong to
		$\mathbb{Q}z_j^{-1}+\mathbb{Q}+\mathbb{Q}z_j$. This completes the proof.
	\end{proof}
	
	\section{Derivative of $\xi_\L(s)$ at the origin}
	\begin{lemma}\label{hankel_bose_integral}
		For $A>0$, we have
		$$\FP \int_{0}^{\infty} \frac{1}{e^{Az}-1}\frac{dz}{z}
		=\frac{1}{2}\left(\gamma+\log A - \log(2\pi)\right).$$
	\end{lemma}
	
	\begin{proof}
		Corollary \ref{corollary_HDF_1dim} implies that the stated Hadamard finite part equals the constant term in the Laurent expansion at $s=0$ of the function $M_{A}(s)$ defined by 
		$$M_{A}(s):=\int_0^\infty \frac{z^{s-1}dz}{e^{Az}-1}, \qquad \Re(s)>1.$$
		Since $M_A(s)=A^{-s}\Gamma(s)\zeta(s)$, its meromorphic continuation satisfies
		$$M_A(s)=-\frac{1}{2s}+\frac{1}{2}\bigl(\gamma+\log A-\log(2\pi)\bigr)+O(s).$$
		This proves the formula.
	\end{proof}
	
	We now restate and prove Theorem \ref{thm:conical_0_der}. \begin{theorem}\label{derivative_conical_zeta_zero}
		For each one-dimensional flat $X\in L(\A)$, choose a generator $u_X \in V$, defined up to sign, such that $\mathbb{Z} u_X = X\cap\mathbb Z^r$, and put
		$$m_X:=\prod_{\substack{1\leq k\leq N\\ L_k(u_X)\neq 0}}|L_k(u_X)|\in\mathbb Q_{>0}.$$
		Then
		$$\xi_\L'(0)=N\mu_\A(V,\{0\})\log(2\pi)+\sum_{\substack{X\in L(\A)\\\dim X=1}}\mu_\A(V,X)\log m_X.$$
	\end{theorem}
	
	\begin{proof}
		Recall our notation $G_\L(\y):=\sum_{D\in\Ch(\A)}\mathrm H_D(l_D)$,
		so that $I_{G_\L}(s)=\xi_\L(s)$. Proposition~\ref{prop:I-values}, applied at
		$n=0$, and Theorem~\ref{values_at_nonpositive_even} give
		\begin{equation}\label{xi_derivative_start}
			\xi_\L'(0) =\frac{1}{N!}\sum_{\sigma\in\mathfrak S_N}
			\sum_{j=1}^N(N-j+1)\mathcal I_{\sigma,j,0}
			+N\gamma\mu_\A(V,\{0\}),\end{equation}
		where 
		$$\mathcal{I}_{\sigma,j,0} = \FP\int_{0}^{R_j} \Res{z_k=0 (k\ne j)}\pi^\ast\left(G_\L^\sigma(\y) \frac{d\y}{\y}\right),\qquad R_1=\infty, \; R_{j}=1 \; (j\ge 2).$$
		First we claim that \begin{equation}\label{aux_2}\sum_{\sigma\in\mathfrak S_N} \mathcal{I}_{\sigma,j,0} = 0,\qquad j\geq 2.\end{equation} To see this, write $$\mathscr{R}_j(G_\mathcal{L}^\sigma) = \Res{z_k=0 (k\neq j)}\pi^\ast\left(G_\L^\sigma(\y) \frac{d\y}{\y}\right) = h_{\sigma,j}(z_j)\frac{dz_j}{z_j}.$$
		Let $\tau_j=(j-1\;j) \in \mathfrak{S}_N$ be the $2$-cycle that swaps $j$ and $j-1$. Then swapping $y_{j-1}$ and $y_j$ is equivalent to changing $z_j$ to $z_j^{-1}$ and multiplying the neighboring $\z$-variables by $z_j$; the latter does not affect the residues in those variables. Hence
		$$h_{\tau_j \sigma,j}(z)=h_{\sigma,j}(z^{-1}).$$
		Moreover, by Proposition~\ref{laurent_partial_residue}, $h_{\sigma,j}(z)$ is of the form
		$h_{\sigma,j}(z)=c_{-1}z^{-1}+c_0+c_1z \ (c_i\in \mathbb{Q}).$
		By the definition of the Hadamard finite part,
		$$\mathcal{I}_{\sigma,j,0} = \FP \int_0^1 h_{\sigma,j}(z)\frac{dz}{z}=-c_{-1}+c_1.$$
		Replacing $z$ by $z^{-1}$ negates this sum, so $\mathcal I_{\tau_j \sigma,j,0}=-\mathcal I_{\sigma,j,0}$.
		However, 
		$$
		\sum_{\sigma\in\mathfrak S_N}\mathcal I_{\sigma,j,0}=\sum_{\sigma\in\mathfrak S_N}\mathcal I_{\tau_j \sigma,j,0},
		$$
		since both are summing over all elements of $\mathfrak{S}_N$. Proving our claim \eqref{aux_2}, therefore equation \eqref{xi_derivative_start} becomes
		\begin{equation}\label{aux_3}\xi_\L'(0)=\frac{N}{N!}\sum_{\sigma\in\mathfrak S_N}\mathcal I_{\sigma,1,0}+N\gamma\mu_\A(V,\{0\}).\end{equation}
		
		Write $k=\sigma^{-1}(1)$, recall the vector $u_{X,k}\in V$ defined before Proposition \ref{one_sided_residue}. Applying that proposition gives
		$$\begin{aligned}\label{first_logarithmic_integral}
			\mathcal I_{\sigma,1,0}
			&=\sum_{\substack{X\in L(\A)\\\dim X=1,\ X\not\subset H_k}} 2\mu_\A(V,X) \left(\FP\int_{0}^{\infty}\frac{1}{e^{z_1 L_k(u_{X,k})}-1}\frac{dz_1}{z_1}\right) \\
			&= \sum_{\substack{X\in L(\A)\\\dim X=1,\ X\not\subset H_k}} \mu_\A(V,X) \left(\gamma+\log L_k(u_{X,k})-\log(2\pi)\right),
		\end{aligned}$$
		where we also used Lemma \ref{hankel_bose_integral} for evaluation of the last integral. There are $(N-1)!$ permutations satisfying $\sigma^{-1}(1)=k$, \eqref{aux_3} then becomes
		\begin{equation}\label{aux_4}
			\xi_\L'(0) = \sum_{k=1}^N
			\sum_{\substack{X\in L(\A)\\\dim X=1,\ X\not\subset H_k}}\mu_\A(V,X) \left(\gamma+\log L_k(u_{X,k})-\log(2\pi)\right) + N \gamma \mu_{\A}(V,\{0\}).
		\end{equation}
		The identity $$\sum_{\substack{X\in L(\A)\\\dim X=1,\ X\not\subset H_k}}\mu_\A(V,X)=-\mu_\A(V,\{0\})$$
		from Lemma \ref{Weisner_lemma} says that the two terms involving Euler's constant cancel in \eqref{aux_4}, and
		\begin{equation*}\xi_\L'(0)  = N \mu_\A(V,\{0\}) \log(2\pi) + \sum_{\substack{X\in L(\A)\\\dim X=1}}
			\mu_\A(V,X)\log \left(\prod_{\substack{1\leq k\leq N\\X\not\subset H_k}}|L_k(u_{X,k})|\right).
		\end{equation*}
		Since $u_{X,k}\in \{\pm u_{X}\}$ for each $k$ with $X\not\subset H_k$, the argument of the logarithm is precisely $m_X$. This proves the theorem.
	\end{proof}

	\begin{example}\label{Ex_2}
		Let $V=\mathbb R^3$ and consider $$\xi_\L(s) := \sum_{\boldsymbol{x} \in \mathbb{Z}^3}^{} {'} \frac{1}{\left|x_1x_2x_3(2x_1+x_2)(x_1+2x_2+x_3)(x_1+2x_3)\right|^{s}},$$ where $\L=(L_1,\ldots,L_6)$ is given by
		$$L_1=x_1,\quad L_2=x_2,\quad L_3=x_3,\quad L_4=2x_1+x_2,\quad L_5=x_1+2x_2+x_3,\quad L_6=x_1+2x_3.$$
		Writing $H_i=\ker L_i$, the eleven one-dimensional flats, together with primitive generators $u_X$ and the data entering Theorem~\ref{derivative_conical_zeta_zero}, are as follows:
		$$
		\begin{array}{c|c|c|c}
			X & u_X & \mu_\A(V,X) & m_X \\
			\hline
			H_1\cap H_2\cap H_4 & (0,0,1) & 2 & 2 \\
			H_1\cap H_3\cap H_6 & (0,1,0) & 2 & 2 \\
			H_1\cap H_5 & (0,1,-2) & 1 & 8 \\
			H_2\cap H_3 & (1,0,0) & 1 & 2 \\
			H_2\cap H_5 & (1,0,-1) & 1 & 2 \\
			H_2\cap H_6 & (2,0,-1) & 1 & 8 \\
			H_3\cap H_4 & (1,-2,0) & 1 & 6 \\
			H_3\cap H_5 & (2,-1,0) & 1 & 12 \\
			H_4\cap H_5 & (1,-2,3) & 1 & 42 \\
			H_4\cap H_6 & (2,-4,-1) & 1 & 56 \\
			H_5\cap H_6 & (4,-1,-2) & 1 & 56
		\end{array}
		$$
		For example, for $X=H_4\cap H_5$, we may take $u_X=(1,-2,3)$, for which
		$$(L_1(u_X),\ldots,L_6(u_X))=(1,-2,3,0,0,7),$$
		and hence $m_X=1\cdot2\cdot3\cdot7=42$. This flat exhibits all the "bad primes" $2,3,7$. Also $\mu_A(V,\{0\}) = -8$. Therefore Theorem~\ref{derivative_conical_zeta_zero} gives
		$$\xi_\L'(0)=-48\log(2\pi) +22\log 2 + 3\log 3 + 3\log 7.$$
		In the attachment, we provide a SageMath code, which makes use of its built-in hyperplane arrangement functionalities, to calculate $\xi_\L(0)$ and $\xi_\L'(0)$ for general $\L$.
	\end{example}
	
	We now restate and prove Corollary \ref{cor:intro-bad-primes}.
	\begin{corollary}
		Assume the linear forms $L_1,\ldots,L_N$ have integral coefficients. Let $A$ be the $N\times r$ integral matrix obtained from their coefficients. We say a prime $p$ is bad if $p$ divides one of its non-zero $r\times r$ minors. Then
		$$\xi_\L'(0) \in \mathbb{Z} \log(2\pi) + \bigoplus_{p \text{ bad}} \mathbb{Z}\log p.$$
	\end{corollary}
	\begin{proof}
		Let $X$ be a one-dimensional flat given by $H_{i_1}\cap \cdots \cap H_{{i_{r-1}}}$. It suffices to show that $L_k(u_X) \in \mathbb{Q}^\times$ is a $p$-adic unit whenever $p$ is good and $X\not\subset H_k$. This follows from a straightforward $p$-adic argument. \par
		More precisely, let $P_p:= \mathbb{Z}^r \otimes_{\mathbb Z}\mathbb Z_p$. We regard the rows of $A$, equivalently the forms $L_1,\ldots,L_N$, as elements of the dual $P_p^\ast = \operatorname{Hom}(P_p,\mathbb{Z}_p)$. Since $L_{i_1},\ldots,L_{i_{r-1}},L_k$ are linearly independent, the determinant of their coefficient matrix is therefore a nonzero $r\times r$ minor of $A$, which is a $p$-adic unit by assumption, hence they form a
		$\mathbb Z_p$-basis of $P_p^\ast$. Therefore there exists $v\in P_p$ such that
		$$L_{i_j}(v)=0\quad(1\leq j\leq r-1),\qquad L_k(v)=1.$$
		The first set of equalities shows that
		$v\in X\otimes_{\mathbb Q}\mathbb Q_p$. Also, from the definition of $u_X$, we have
		$$P_p\cap\bigl(X\otimes_{\mathbb Q}\mathbb Q_p\bigr)=\mathbb Z_pu_X.$$
		It follows that $v=a u_X$ for some $a\in\mathbb Z_p$. Therefore $1=L_k(v)=aL_k(u_X).$
		Both factors belong to $\mathbb Z_p$, so $L_k(u_X)$ is a $p$-adic unit, as desired.
	\end{proof}
	
	\section{Derivative of Witten zeta function at origin}
	Adopting notations associated to root systems in Section \ref{root_system_subsection}. For each positive integer $d$, recall the number $m_\Phi(d)$ defined in the introduction $$m_\Phi(d) := \sum_{i=1}^r \frac{E(\Phi_{\Delta-\{\alpha_i\}})}{|W_{\Delta-\{\alpha_i\}}|} \#\{ \gamma\in \Phi^+ : (\lambda_i, \gamma^\vee) = d\}.$$
	Note that $m_\Phi(d) \neq 0$ for only finitely many $d$. More precisely, if the highest root of $\Phi$ is $c_1 \alpha_1 + \cdots + c_r \alpha_r,$
	then $m_\Phi(d) = 0$ for $d > \max(c_1,\ldots,c_r)$.

	\begin{proof}[Proof of Theorem \ref{thm:witten_0_der}]
		Our goal is to prove $$\xi_\Phi'(0) = (-1)^r \frac{N E(\Phi)}{|W|} \log(2\pi) + \frac{(-1)^{r+1}}{2} \sum_{d\geq 1} m_\Phi(d) \log d.$$
		Theorem \ref{derivative_conical_zeta_zero} gives
		\begin{equation}\label{aux_8}\xi_\Phi'(0)= (-1)^r \frac{N E(\Phi)}{|W|} \log(2\pi)+ \frac{1}{|W|}\sum_{\substack{X\in L(\A)\\\dim X=1}}\mu_\A(V,X)\log m_X,\end{equation}
		where $$m_X=\prod_{\substack{\gamma\in \Phi^+\\(u_X,\gamma^\vee)\ne0}}|(u_X,\gamma^\vee)|$$
		and $u_X$ is a primitive lattice generator of $X$.
		Write $J_i:=\Delta-\{\alpha_i\}$. By Lemma \ref{aux_lemma}, $X$ is of the form $w\lambda_i$ and the stabilizer of $\lambda_i$ is $W_{J_i}$. Put $X_i := \mathbb{R}\lambda_i$, then
		\begin{equation}\label{aux_7}\sum_{\substack{X\in L(\A)\\\dim X=1}}
			\mu_\A(V,X)\log m_X=\frac12\sum_{i=1}^r\sum_{w\in W/W_{J_i}}\mu_\A(V,wX_i)\log m_{wX_i},\end{equation}
		here the factor $\frac12$ arises because $\mathbb{R}w\lambda_i$ and $-\mathbb{R}w\lambda_i$ define the same flat. The Weyl group preserves the intersection lattice, so Lemma~\ref{para_exp} gives
		$$\mu_\A(V,wX_i)=(-1)^{r-1}E(\Phi_{J_i}),\qquad w\in W.$$
		Moreover, the weight lattice is preserved by $W$. Thus $w\lambda_i$ is a primitive lattice generator of
		$wX_i$. We consequently have
		$$m_{wX_i}=\prod_{\substack{\gamma\in\Phi^+ \\ (w\lambda_i,\gamma^\vee)\neq0}}
		\left|(w\lambda_i,\gamma^\vee)\right| =\prod_{\substack{\gamma\in\Phi^+\\
				(\lambda_i,\gamma^\vee)>0}}(\lambda_i,\gamma^\vee):= M_i.$$
		Therefore the inner sum on the RHS of \eqref{aux_7} is independent of $w$, so
		$$\sum_{\substack{X\in L(\A)\\\dim X=1}}
		\mu_\A(V,X)\log m_X
		=\frac{(-1)^{r-1} |W| }{2}\sum_{i=1}^r\frac{E(\Phi_{J_i})}{|W_{J_i}|}\log M_i.$$
		Now,
		$$\log M_i=\sum_{\substack{\gamma\in\Phi^+\\(\lambda_i,\gamma^\vee)>0}}
		\log(\lambda_i,\gamma^\vee)\\=\sum_{d\geq1}
		\#\{\gamma\in\Phi^+:(\lambda_i,\gamma^\vee)=d\}\log d.$$
		By the definition of $m_\Phi(d)$, this proves
		$$\frac{1}{|W|}\sum_{\substack{X\in L(\A)\\\dim X=1}}\mu_\A(V,X)\log m_X
		=\frac{(-1)^{r+1}}{2}\sum_{d\geq1}m_\Phi(d)\log d.$$
		Plugging this into \eqref{aux_8} completes the proof.
	\end{proof}
	
	Although not necessary for our argument on finding $\xi_\Phi'(0)$, we present an interesting property of the numbers $m_\Phi(d)$.
	\begin{corollary}\label{m_Phi_sum_coro}
		We have
		$$\sum_{d\geq 1}m_\Phi(d) = \frac{|\Phi| E(\Phi)}{|W|}.$$
	\end{corollary}
	\begin{proof}
		As in the previous proof, let $J_i:=\Delta-\{\alpha_i\}$, also let $N_i := |\Phi_{J_i}^+|$. Then
		$$\sum_{d\geq1}\#\{\gamma\in\Phi^+:(\lambda_i,\gamma^\vee)=d\}=N-N_i.$$
		Hence \begin{equation}\label{aux_9}\sum_{d\geq1}m_\Phi(d)=
			\sum_{i=1}^r\frac{E(\Phi_{J_i})}{|W_{J_i}|}(N-N_i).\end{equation}
		For every reflecting hyperplane $H_\gamma, \gamma\in \Phi^+$, Lemma~\ref{Weisner_lemma} gives
		$$\sum_{\substack{\dim X=1\\X\not\subset H_\gamma}}
		\mu_{\mathcal A}(V,X)=-\mu_{\mathcal A}(V,\{0\}).$$
		Summing over all $\gamma \in \Phi^+$ and multiplying by $(-1)^{r-1}$ gives
		$$NE(\Phi)= \sum_{\dim X=1}|\mu_{\mathcal A}(V,X)|
		\#\{\gamma\in\Phi^+:X\not\subset H_\gamma\}.$$
		Every one-dimensional flat $X$ is of the form $w\lambda_i$, and the same argument leading to \eqref{aux_7} gives
		$$NE(\Phi) = \frac{1}{2} \sum_{i=1}^r\sum_{w\in W/W_{J_i}} |\mu_{\mathcal A}(V,\mathbb{R}w\lambda_i)|
		\#\{\gamma\in\Phi^+: w\lambda_i \notin H_\gamma\}.$$
		Since $|\mu_{\mathcal A}(V,\mathbb{R}w\lambda_i)| = E(\Phi_{J_i})$ and $$\#\{\gamma\in\Phi^+: w\lambda_i \notin H_\gamma\} = \#\{\gamma\in\Phi^+: \lambda_i \notin H_\gamma\} = N - N_i,$$
		the above becomes
		$$NE(\Phi) = \frac{|W|}{2}\sum_{i=1}^r \frac{E(\Phi_{J_i})}{|W_{J_i}|}(N-N_i).$$
		Comparing with \eqref{aux_9} completes the proof.
	\end{proof}
	
	\subsection{Values of $m_\Phi(d)$ for each irreducible root system}
	Finally, we shall find the values of $m_\Phi(d)$ for each irredducible root system, which are recorded in Table \ref{table:m-Phi-values}.
	\begin{table}[h]
		\centering
		\renewcommand{\arraystretch}{1.45}
		\resizebox{0.8\textwidth}{!}{%
			\begin{tabular}{c|cccccc}
				\hline
				$\Phi$ & $m_\Phi(1)$ & $m_\Phi(2)$ & $m_\Phi(3)$ & $m_\Phi(4)$ & $m_\Phi(5)$ & $m_\Phi(6)$ \\
				\hline
				$A_r$
				& $r$
				& \cellcolor{gray!20}\textcolor{gray}{$-$}
				& \cellcolor{gray!20}\textcolor{gray}{$-$}
				& \cellcolor{gray!20}\textcolor{gray}{$-$}
				& \cellcolor{gray!20}\textcolor{gray}{$-$}
				& \cellcolor{gray!20}\textcolor{gray}{$-$} \\
				$B_r$
				& {\footnotesize $\displaystyle \frac{r+1}{2}+\frac{2(r-1)}{3}a_r$}
				& {\footnotesize $\displaystyle -\frac{r+1}{2}+\frac{r+2}{3}a_r$}
				& \cellcolor{gray!20}\textcolor{gray}{$-$}
				& \cellcolor{gray!20}\textcolor{gray}{$-$}
				& \cellcolor{gray!20}\textcolor{gray}{$-$}
				& \cellcolor{gray!20}\textcolor{gray}{$-$} \\
				$C_r$
				& {\footnotesize $\displaystyle \frac{2r+1}{3}a_r$}
				& {\footnotesize $\displaystyle \frac{r-1}{3}a_r$}
				& \cellcolor{gray!20}\textcolor{gray}{$-$}
				& \cellcolor{gray!20}\textcolor{gray}{$-$}
				& \cellcolor{gray!20}\textcolor{gray}{$-$}
				& \cellcolor{gray!20}\textcolor{gray}{$-$} \\
				$D_r$
				& {\footnotesize $\displaystyle r-1+\frac{2(r-2)}{3}a_{r-1}$}
				& {\footnotesize $\displaystyle 1-r+\frac{r+1}{3}a_{r-1}$}
				& \cellcolor{gray!20}\textcolor{gray}{$-$}
				& \cellcolor{gray!20}\textcolor{gray}{$-$}
				& \cellcolor{gray!20}\textcolor{gray}{$-$}
				& \cellcolor{gray!20}\textcolor{gray}{$-$} \\
				$G_2$
				& $3$ & $1$ & $1$
				& \cellcolor{gray!20}\textcolor{gray}{$-$}
				& \cellcolor{gray!20}\textcolor{gray}{$-$}
				& \cellcolor{gray!20}\textcolor{gray}{$-$} \\
				$F_4$
				& $\frac{79}{8}$ & $5$ & $\frac{2}{3}$ & $\frac{1}{2}$
				& \cellcolor{gray!20}\textcolor{gray}{$-$}
				& \cellcolor{gray!20}\textcolor{gray}{$-$} \\
				$E_6$
				& $\frac{46}{3}$ & $\frac{5}{3}$ & $\frac{1}{9}$
				& \cellcolor{gray!20}\textcolor{gray}{$-$}
				& \cellcolor{gray!20}\textcolor{gray}{$-$}
				& \cellcolor{gray!20}\textcolor{gray}{$-$} \\
				$E_7$
				& $\frac{1295}{48}$ & $\frac{2713}{512}$ & $\frac{5}{6}$ & $\frac{1}{8}$
				& \cellcolor{gray!20}\textcolor{gray}{$-$}
				& \cellcolor{gray!20}\textcolor{gray}{$-$} \\
				$E_8$
				& $\frac{57673}{1152}$ & $\frac{155443}{9216}$ & $\frac{1643}{324}$
				& $\frac{55}{32}$ & $\frac{2}{5}$ & $\frac{1}{6}$ \\
				\hline
			\end{tabular}%
		}
		\caption{\small The values of $m_\Phi(d)$ for all irreducible root systems. Here
			$a_k:=\displaystyle \frac{k}{2^{2k-1}}\binom{2k}{k}$.}
		\label{table:m-Phi-values}
	\end{table}
	
	For type $A_r$, this follows directly from Corollary \ref{m_Phi_sum_coro}. For the five exceptional types, it is a simple computation (see the attached SageMath code). Therefore we focus on types $B_r, C_r, D_r$. 
	
	For $k\in\mathbb{Z}_{\ge 1}$, put $a_k:=\frac{k}{2^{2k-1}}\binom{2k}{k}$. We use the elementary identities
	$$\sum_{k=1}^K\frac{a_k}{k}=2a_{K+1}-2,
	\qquad\sum_{k=1}^Ka_k=\frac{2K}{3}a_{K+1},$$
	which follow immediately from $a_{k+1}=(2k+1)a_k/(2k)$ by telescoping.
	We also recall that
	$$
	\frac{E(A_q)}{|W(A_q)|}=\frac1{q+1},\qquad \frac{E(B_q)}{|W(B_q)|}=\frac{E(C_q)}{|W(C_q)|}=\frac{a_q}{2q},
	\qquad \frac{E(D_q)}{|W(D_q)|}=\frac{a_{q-1}}{2q}.$$
	
	We adopt the standard realization in $\mathbb{R}^r$
	with basis $\varepsilon_1,\ldots,\varepsilon_r$, where
	$$\alpha_i=\varepsilon_i-\varepsilon_{i+1}\qquad(1\le i<r),$$
	and $\alpha_r=\varepsilon_r,2\varepsilon_r,\varepsilon_{r-1}+\varepsilon_r$ for types $B_r,C_r,D_r$, respectively.
	For these three types, $(\lambda_i,\gamma^\vee)\in\{0,1,2\}$. Hence, by Corollary~\ref{m_Phi_sum_coro}, it remains only to determine $m_\Phi(2)$.
	
	\begin{itemize}[leftmargin=*]
		\item For $B_r$,
		\[
		\Phi^+=\{\varepsilon_i\pm\varepsilon_j:1\le i<j\le r\}
		\cup\{\varepsilon_i:1\le i\le r\}
		\]
		and
		\[
		\lambda_i=\varepsilon_1+\cdots+\varepsilon_i \quad (1\le i\le r-1),
		\qquad \lambda_r=\frac12 (\varepsilon_1+\cdots+\varepsilon_r).
		\]
		Hence
		\[
		\#\{\gamma\in\Phi^+:(\lambda_i,\gamma^\vee)=2\}
		=\frac{i(i+1)}2
		\quad(1\le i\le r-1)
		\]
		and the number is zero for $i=r$.
		Moreover $\Phi_{\Delta - \{\alpha_1\}}\cong B_{r-1}$,
		$\Phi_{\Delta-\{\alpha_i\}}\cong A_{i-1}\times B_{r-i}$ for
		$1<i<r$,
		and $\Phi_{\Delta- \{\alpha_r\}} \cong A_{r-1}$.
		Thus
		\[
		m_{B_r}(2)
		=\sum_{i=1}^{r-1}\frac{i+1}{4(r-i)}a_{r-i}
		=\sum_{k=1}^{r-1}\frac{r-k+1}{4k}a_k
		=\frac{r+2}{3}a_r-\frac{r+1}{2}.
		\]
		
		\item For $C_r$,
		\[
		\Phi^+=\{\varepsilon_i\pm\varepsilon_j:1\le i<j\le r\}
		\cup\{2\varepsilon_i:1\le i\le r\}
		\]
		and
		$\lambda_i=\varepsilon_1+\cdots+\varepsilon_i$ ($1\le i\le r$).
		Since
		\[
		\#\{\gamma\in\Phi^+:(\lambda_i,\gamma^\vee)=2\}
		=\binom{i}{2} \quad (1\le i\le r),
		\]
		we obtain
		\[
		m_{C_r}(2)
		=\sum_{i=1}^{r-1}\frac{i-1}{4(r-i)}a_{r-i}+\frac{r-1}{2}
		=\frac{r-1}{3}a_r.
		\]
		
		\item For $D_r$,
		we have $\Phi^+=\{\varepsilon_i\pm\varepsilon_j:1\le i<j\le r\}$,
		and
		$\lambda_i=\varepsilon_1+\cdots+\varepsilon_i$ ($1\le i\le r-2$),
		while
		\[
		\lambda_{r-1}
		=\frac12(\varepsilon_1+\cdots+\varepsilon_{r-1}-\varepsilon_r),
		\qquad
		\lambda_r
		=\frac12(\varepsilon_1+\cdots+\varepsilon_{r-1}+\varepsilon_r).
		\]
		Since
		\[
		\#\{\gamma\in \Phi^+:(\lambda_i,\gamma^\vee)=2\}
		=\binom{i}{2} \qquad (1<i<r-1)
		\]
		and the number is zero for $i=1,r-1,r$,
		we obtain
		$$m_{D_r}(2)
		=\sum_{i=2}^{r-2}\frac{(i-1)a_{r-i-1}}{4(r-i)} = \frac{r+1}{3}a_{r-1}-r+1.$$
	\end{itemize}
	
	\bibliographystyle{plain} 
	\bibliography{conical_ref.bib} 
	
\end{document}